\documentclass[reqno]{amsart}

\usepackage{amsfonts}
\usepackage{amscd}
\usepackage{amsbsy}
\usepackage{amsxtra}
\usepackage{amssymb}
\usepackage{epsfig}
\usepackage{epic,eepic}
\usepackage{graphicx}
\usepackage[all]{xy}
\usepackage{xypic}
\usepackage{psfrag}
\usepackage{amsmath}
\usepackage{amsthm}
\usepackage{setspace}
\usepackage{hyperref}
\usepackage{url}
\usepackage{here}
\usepackage{todonotes}
\usepackage{pdflscape}
\newlength{\halfbls}

   \newcommand{\CC}{\mathbb{C}}
  \newcommand{\HH}{\mathbb{H}}
   
 \newcommand{\WW}{\mathbb{W}}  \newcommand{\VV}{\mathbb{V}}
 \newcommand{\PP}{\mathbb{P}}  \newcommand{\NN}{\mathbb{N}}
\newcommand{\QQ}{\mathbb{Q}} \newcommand{\RR}{\mathbb{R}} 

\newcommand{\ZZ}{\mathbb{Z}}

\newcommand{\cLL}{{\mathcal L}} 
 \newcommand{\cFF}{{\mathcal F}} \newcommand{\cEE}{{\mathcal E}}

\newcommand{\cHH}{{\mathcal H}}
\newcommand{\cMM}{{\mathcal M}} \newcommand{\cOO}{{\mathcal O}}
 \newcommand{\cVV}{{\mathcal V}}

\newcommand{\bfa}{{\bf a}}
\newcommand{\bfb}{{\bf b}}

\newcommand{\bfal}{{\boldsymbol{\alpha}}}
\newcommand{\bfbe}{{\boldsymbol{\beta}}}

 \newcommand{\Sp}{{\rm Sp}}

\newcommand{\gr}{{\rm gr}}
     
 \newcommand{\GL}{{\rm GL}}   

\newcommand{\adm}{admissible}

\newcommand{\rank}{{\rm rank}}

\DeclareMathOperator{\vol}{vol}

\DeclareMathOperator{\Ann}{Ann}

\newcommand{\sms}{\smallsetminus}    

    \newcommand{\ve}{{\varepsilon}}
\newcommand{\ol}{\overline}

\newtheorem{Defi}{Definition}[section]  
\newtheorem{Rem}[Defi]{Remark}    \newtheorem{Prop}[Defi]{Proposition}
\newtheorem{Lemma}[Defi]{Lemma}    \newtheorem{Cor}[Defi]{Corollary}
\newtheorem{Thm}[Defi]{Theorem}  
\newtheorem*{nnThm}{Theorem}  
\newtheorem*{nnConj}{Conjecture}  
  
\newtheorem{Conj}[Defi]{Conjecture}


\def\={\;=\;}  \def\+{\,+\,}        
         
   \def\wt#1{\widetilde{#1}}

\newcommand\R{\mathbb R}  \newcommand\C{\mathbb C}

       \def\p{\varphi}   \def\2{\pi_2}

\def\be{\begin{equation}}   \def\ee{\end{equation}}     \def\bes{\begin{equation*}}    \def\ees{\end{equation*}}
\def\ba{\be\begin{aligned}} \def\ea{\end{aligned}\ee}   \def\bas{\bes\begin{aligned}}  \def\eas{\end{aligned}\ees}

\newcommand{\dd}{\,\mathrm{d}}
\newcommand{\hyp}{{\rm hyp}}

\newcommand{\Sol}{{\rm Sol}}

\newcommand{\bad}{{\rm bad}}

\newcounter{savedtocdepth}
\newcommand*{\SaveTocDepth}[1]{%
  \addtocontents{toc}{%
    \protect\setcounter{savedtocdepth}{\protect\value{tocdepth}}%
    \protect\setcounter{tocdepth}{#1}%
  }%
}

\usepackage[style=alphabetic,
            isbn=false,
            doi=false,
            url=false, sorting = nyt,
            backend=bibtex, maxnames = 5
           ]{biblatex}
\DeclareFieldFormat[article]{title}{{\it #1}} 
\DeclareFieldFormat{journaltitle}{{\rm #1}} 
\renewbibmacro{in:}{%
  \ifentrytype{article}{}{\printtext{\bibstring{in}\intitlepunct}}}

\bibliography{my}

\title[Lower bounds for Lyapunov exponents]
{Lower bounds for Lyapunov exponents of \\ flat bundles on curves}

\author[A.~Eskin]{Alex Eskin}
\thanks{Research  of  the first author is partially supported  by
NSF grant.}
\address{
Department of Mathematics,
University of Chicago,
Chicago, Illinois 60637, USA\\
}
\email{eskin@math.uchicago.edu}

\author[M.~Kontsevich]{Maxim Kontsevich}
\address{IHES,
le Bois Marie,
35, route de Chartres,
91440 Bures-sur-Yvette, FRANCE\\
}
\email{maxim@ihes.fr}

\author[M.~M\"oller]{Martin M\"oller}
\thanks{Research  of the third author is partially supported  by
DFG-grant MO 1884/1.}
\address{
Institut f\"ur Mathematik, Goethe--Universit\"at Frankfurt,
Robert-Mayer-Str. 6--8,
60325 Frankfurt am Main, Germany\\
}
\email{moeller@math.uni-frankfurt.de}

\author[A.~Zorich]{Anton Zorich}
\thanks{Research of the fourth author is partially supported by IUF}
\address{
Center for Advanced Studies, Skoltech;
Institut Universitaire de France;
Institut de Math\'ematiques de Jussieu --
Paris Rive Gauche,
B\^atiment Sophie Germain,
Case 7012,
8 Place Aurélie Nemours,
75205 PARIS Cedex 13, France}
\email{anton.zorich@gmail.com}

\date{August 17, 2017}
\dedicatory{To the memory of Jean-Christophe Yoccoz}

\begin{document}

\begin{abstract} Consider a flat bundle over a complex curve.
We prove a conjecture of Fei Yu that the sum of the top~$k$~Lyapunov
exponents of the flat bundle is always greater or equal to the degree of 
any rank~$k$
holomorphic subbundle. We generalize the original context
from Teichm\"uller curves to any local system over a curve with
non-expanding cusp monodromies. As an application we obtain the large genus
limits of individual Lyapunov exponents in hyperelliptic strata of
Abelian differentials, proved by Fei Yu conditionally to his conjecture.
\par
Understanding the case of equality with the degrees of subbundle coming from 
the Hodge filtration seems challenging, e.g.\ for Calabi--Yau type families. 
We conjecture that equality of the sum of Lyapunov exponents and the degree 
is related to the monodromy group being a thin subgroup of its Zariski closure.
\end{abstract}

\maketitle

\noindent
\SaveTocDepth{1}

\section{Introduction} 
\label{sec:intro}

Lyapunov exponents are dynamical analogs of characteristic numbers of vector bundles.
The Lyapunov exponents for the Teichm\"uller geodesic flow relate the
dynamics on moduli space with the dynamics on flat surfaces. Efficiently
computing them is currently still a challenge, both for strata of the
moduli space of flat surfaces and for Teichm\"uller curves, including
all the Teichm\"uller curves generated by square-tiled surfaces. Starting 
with \cite{kontsevich} it was realized that the {\em sum} of
(i.e.\  the sum of 
the positive) Lyapunov exponents equals the normalized degree of the Hodge 
bundle on Teichm\"uller 
curves, see \cite{forni02}, \cite{krik}, \cite{bouwmoel}, \cite{eskozo} for versions of 
this formula, including the case of strata. This observation generalizes
from the variation of Hodge structures over Teichm\"uller curves to
any weight one variation of Hodge structures (VHS). Presently, irreducible
summands of weight one VHS are the only instances where such degree
formulas are known. Even the computation of Filip (\cite{Filip:K3}) of the top 
Lyapunov exponent for families of K3 surfaces can be subsumed under this observation,
if one refers to his proof using the Kuga--Satake construction.
\par
The main result of this paper is that an {\em inequality} for the
sum of the top~$k$ Lyapunov exponents holds in great generality. This was first 
conjectured by \cite{Fei}, but the scope given here is more general. 
\par
Let $C=\HH/\Gamma$ be hyperbolic Riemann surface of finite area (or equivalently, 
a complex quasi-projective curve) with a representation $\rho: \pi_1(C) \to \GL(V)$
such that, if $C$ is non-compact, the mondromies around the cusps $\Delta
= \ol{C} \setminus C$ are
{\em non-expanding}, i.e.\ all the eigenvalues lie on the unit circle. This
assumption is  necessary  and also sufficient for Oseledets theorem, see 
Sections~\ref{sec:pfintegrab} and~\ref{sec:nec}.
To be more precise, we need to specify a norm on the flat bundle~$\VV$ determined
by~$\rho$. There are two natural choices: the practical choice (for simulations)
is a ``constant'' norm obtained by parallel transport along a 
Dirichlet fundamental domain
for~$\Gamma$ and the sophisticated choice of an {\em admissible} norm 
(see Section~\ref{sec:admmetric} for the precise definition) that has the right growth 
at the cusps and compatibility with exterior powers. Oseledets theorem is very
insensitive to such choices: 
we show (Theorem~\ref{thm:admisintegrable} and Proposition~\ref{prop:constantmetric}, 
see also the appendix for the background on measurable cocycles) 
that both norms satisfy the integrability condition and compute the same Lyapunov 
exponents.
\par
For VHS of arbitrary weight we show that the Hodge norm is admissible. Along 
with the proof (Proposition~\ref{prop:HodgeAdm}) we give an upper bound for the Lyapunov
exponents that is uniform for all VHS of given weight and rank.
However, our estimate is very crude. It is an interesting problem to prove tight 
upper bounds for Lyapunov exponents for VHS.
\par
In the setting of a local system $\VV$ defined by $\rho$ and a norm as above,
we can now state our main Theorem~\ref{thm:main_estimate}:
\begin{nnThm} 
For any holomorphic rank~$k$ subbundle~$\cEE$ of the Deligne extension
of $\VV$ the sum of the top~$k$ Lyapunov exponents is bounded below by
\be
\label{intro:main:ineqver2}
\sum_{i=1}^k       \lambda_i       \,\geq\,
\frac{2\deg_{\rm par}(\cEE)}{2g(\ol{C})-2+|\Delta|}
\,, \ee
where $g(\ol{C})$ is the genus of the curve~$\ol{C}$ and $|\Delta|$ is the number
of cusps.
\end{nnThm}
\noindent
Here the parabolic degree $\deg_{\rm par}$ of a vector bundle is equal to the 
degree in the case of unipotent monodromies and is defined 
in Section~\ref{sec:parabolic} in general.
\par
A theorem in a similar spirit in rank two was proven previously 
by Deroin and Dujardin in~\cite{derduj}. The main theorem of
the subsequent paper~\cite{dander} by Daniel and Deroin proves a formula simliar
to~\eqref{intro:main:ineqver2} using Brownian motion techniques, 
applicable also to a higher-dimensional base provided that
the base is compact.
\par
\smallskip
This theorem has two types of applications. The first is the large genus
limit of Lyapunov exponents for hyperelliptic strata of Abelian
differentials (Corollary~\ref{cor:hyplim}, 
proven by F.~Yu conditionally to our main theorem).
\par
As preparation for our second application we show in the last section
that the parabolic degrees of the Hodge bundles of
hypergeometric local systems can be easily expressed in terms of the local
exponents,
see Section~\ref{sec:localexp} for the notions and Theorem~\ref{thm:degs} 
for the precise statement. 
\par
This second application concerns families of Calabi--Yau threefolds and 
conjecturally gives new cases where equality in~\eqref{intro:main:ineqver2} holds. 
There is a well-known list of $14$~rank~$4$ hypergeometric local systems
(see Table~\ref{cap:mirror}, including the mirror quintic) that could be the middle 
cohomology of a family of Calabi--Yau threefolds with $h^{2,1}=1$. In
 7 out of these 14 examples the monodromy group is thin in the symplectic group
(see~\cite{BravThomas}, \cite{SinghVenk} and Section~\ref{sec:APPCY}).
\par
\begin{nnConj} The inequality~\eqref{intro:main:ineqver2} becomes an equality 
precisely in the  7 out of these 14 cases where the monodromy group 
is thin.\footnote{Simion Filip has recently
announced a proof of this conjecture.}
\end{nnConj}
\par
Initially, we stated a more optimistic conjecture on a region
in the parameter space for the local exponents where the equality is attained. 
This initial conjecture can no longer be upheld after more detailed numerical
experiments by Fougeron, see \cite{Fougeron}. We discuss this
in more detail in  Section~\ref{sec:APPCY}).
%
\par
\medskip
{\bf Acknowledgements.} The authors thank Fei Yu for the inspiring conjecture;
Simion Filip for enlightning discussions around Proposition~3.2, and the
referee for suggestions which improved the presentation. We also
thank the Max-Planck-Institute for Mathematics in Bonn for its hospitality
during the preparation of the paper.

\section{Lyapunov exponents for flat bundles with non-expanding cusp monodromies} 
\label{sec:LforRegSing}

In this section we show that Lyapunov exponents for flat bundles over
the geodesic flow on a (base) curve are defined for a very large class
of flat bundles. The only restriction that we impose is that the 
monodromies  around the boundary points have eigenvalues of absolute value one. 
This (strictly) includes the case of quasi-unipotent monodromies.
\par
\medskip
We now give the background and the definitions alluded to above. 
Our base manifold will always be an algebraic curve $C = \Gamma \backslash \HH$, 
not necessarily compact. Let $\ol{C}$ be the smooth compactification
and $\Delta = \ol{C} \smallsetminus C$ be the boundary points. The flow
will always be the unit speed geodesic flow $g_t$ on the unit tangent bundle
$T^1 C$ for the metric of constant curvature $-4$ (see Remark~\ref{rem:curvature} 
for the history of this convention) and $\mu$ will be the corresponding invariant 
probability measure.
\par
Let $\VV$ be a flat bundle over $C$ of rank~$r$. We will denote by $(\cVV_C,\nabla)$ 
the  associated vector bundle with its flat connection. 
We say that $\VV$ has {\em non-expanding cusp monodromies}
if for each element $\gamma \in \pi_1(C,c_0)$ homotopic to a simple loop around 
a point in $\Delta$ all the eigenvalues
of $\rho(\gamma)$ have absolute value one. Recall that $\VV$ has {\em 
quasi-unipotent} monodromies if for each element $\gamma \in \pi_1(C,c_0)$ homotopic 
to a simple loop around a point in $\Delta$ there exists some $n$ such
that $\rho(\gamma)^n - {\rm Id}$ is nilpotent. Consequently, having quasi-unipotent
monodromy implies non-expanding cusp monodromies. We show in Section~\ref{sec:pfintegrab}
and~\ref{sec:nec}
that this condition is necessary and sufficient 
for integrability of the flat bundle~$\VV$.
\par
The remaining ingredient we need for the definition of a 
Lyapunov spectrum is a norm $\|\cdot\|$ on $\VV$. We will define in 
Section~\ref{sec:admmetric} a notion of~{\em \adm\ metric~$h$} that we can
provide any local system with and that is suitable for metric extensions
of the line bundle to $\ol{C}$. Such a metric is also the basis to
define Lyapunov exponents for the flat bundle~$\VV$. 
These two notions will be our main hypothesis for the existence of Lyapunov
exponents for flat bundles. Our aim is to
show the following norm bound for the lift $G_t$ of the geodesic flow $g_t$ to $\VV$.
\par
\begin{Thm} \label{thm:admisintegrable}
If $\VV$ is a flat bundle of $\CC$-rank~$r$ 
on $C$ such that the eigenvalues of monodromy 
around points in $\Delta$ all have absolute value one, then for any \adm\ 
metric on~$\VV$ the induced cocycle
is integrable (in the sense of Definition~\ref{def:meascocycle}). 
The corresponding Lyapunov exponents $\lambda_1 \geq \lambda_2 \geq \cdots \geq 
\lambda_r$  are independent of the choice of an \adm\ metric.
\end{Thm}
\par
For practical purposes (e.g.\ for numerical simulations) it is useful to be
able to compute the Lyapunov exponents with a simpler norm. We define a
{\em constant norm $||\cdot||_{\rm const}$} on $\VV$ to be the parallel transport
of any norm at the fiber over some base point~$c_0$ extended to a Dirichlet
fundamental domain for $\Gamma$ on~$\HH$, or equivalently, on a simply connected
complement of some geodesic ``boundary'' curves in~$C$. Note that 
the ``constant'' norm is not continuous across these boundary curves
and depends on the choice of the Dirichlet domain.
\par
\begin{Prop} \label{prop:constantmetric}
Any constant norm $||\cdot||_{\rm const}$ on a flat bundle as in
Theorem~\ref{thm:admisintegrable} is also  integrable
and computes the same Lyapunov exponents as any \adm\ metric.
\end{Prop}

\subsection{Parabolic bundles and filtered vector bundles} \label{sec:parabolic}

We begin with the definition of a parabolic bundle (see also \cite{MeSeshPara}, 
\cite{seshadriPara} for the origins of this notion). 
We first define a {\em $[0,1)$-filtration} on a complex vector space $V$ to 
be a collection of (real) weights $0 \leq \alpha_1 < \alpha_2 < \ldots 
< \alpha_n < \alpha_{n+1}=1$
for some $n \geq 1$ together with a filtration of sub-vector spaces
$$ F^\bullet: \,\, V \= V^{\geq \alpha_1} \supsetneq V^{\geq \alpha_2}\supsetneq 
\cdots \supsetneq V^{\geq \alpha_{n+1}} = V^{\geq 1} = 0.$$
We denote by $\gr_{\alpha_i} V = $ the graded piece at weight $\alpha_i$.  
The {\em filtered dimension} of $(V, F^\bullet)$ is defined to be the real number
$$ \dim_{F^\bullet}(V) \= \sum_{i=1}^n \alpha_i \,\dim \gr_{\alpha_i}(V)\,.$$
The filtration is called {\em trivial}, if $n=1$ and $\alpha_1=0$. This is equivalent
to the condition $ \dim_{F^\bullet}(V) = 0$. 
\par 
Let $\cEE$ be a holomorphic vector bundle on a complex curve $\ol{C}$ and let~$\Delta$
be a finite set of ``boundary'' points. A {\em parabolic structure $(\cEE,F^\bullet)$} 
on $\cEE$ (with respect to $\Delta$) is a $[0,1)$-filtration $F^\bullet \cEE_c$ 
on the fiber $\cEE_c$ for each $c \in \Delta$. A {\em parabolic bundle}
is simply a holomorphic vector bundle with a parabolic structure. 
\par
The {\em parabolic degree} of 
$(\cEE,F^\bullet)$ is defined to be 
$$ \deg_{\rm par}(\cEE, F^\bullet) \= \deg(\cEE) + \sum_{c \in \Delta} \dim_{F^\bullet} \cEE_c\,.$$
\par
A {\em morphism} $\varphi$ between parabolic bundles $\cEE$ and $\cFF$ is
a morphism $\varphi: \cEE \to \cFF$ of holomorphic vector bundles such that for each
$c \in \Delta$ each weight $\alpha$ of $\cEE_c$ the image $\varphi(\cEE_c^{\geq \alpha})$
lies in $\cFF_c^{\geq \beta}$ whenever $\beta \leq \alpha$. A {\em parabolic subbundle}
$\cEE$ of $\cFF$ is an injective morphism of parabolic bundle with the additional
requirements that for each $c \in \Delta$ the weights of $\cEE$ are a subset of
the weights of $\cFF$ and if $\beta$ is maximal such that 
$\varphi(\cEE_c^{\geq \alpha}) \subseteq \cFF_c^{\geq \beta}$ then $\beta = \alpha$.
\par
With this notion of degree and subbundles we will recall later that the
usual notions of stability and of the Harder--Narasimhan filtration carry
over verbatim to the parabolic case. 
\par
For taking exterior powers it will be convenient to use the following equivalent
notion. A {\em filtered vector bundle} $\cEE = \{\cEE_{\bullet,\bullet}\}$
on~$\ol{C}$ is a collection $\cEE_{c,\alpha}$ of vector bundles in
$j_*^c \cEE_C$ for every $c \in \Delta$ and every $\alpha \in \R$
(where $j^c: C \to C \cup \{c\}$ is the inclusion),
such that the filtration is
descending ($\cEE_{c,\alpha} \subseteq \cEE_{c,\beta}$ if $\alpha \geq \beta$),
right continuous ($\cEE_{c, \alpha + \ve} = \cEE_{c,\alpha}$ for small~$\ve$) and 
such that $\cEE_{c, \alpha+1} = t \cEE_{c, \alpha} \subset \cEE_{c, \alpha}$,
where~$t$ is a local parameter at~$c$. 
To retrieve the corresponding bundle with parabolic structure we take the
extensions $\cEE_{c,0}$ at every point $c \in \Delta$ and the filtrations given by the
$\alpha \in [0,1)$ where the rank of the fibers of $\cEE_{c,\alpha}$ at~$c$ jumps. 
In particular, the notions of parabolic degree etc.\ defined 
above apply to filtered vector bundles as well. Obviously a filtered vector bundle
is completely determined by the extensions $\cEE_{c, \alpha}$ for $\alpha \in [0,1)$.
Conversely, given a vector bundle with parabolic structure $(\cEE,F^\bullet)$ we can
provide $\cEE$ with the structure of a filtered bundle $\cEE_{\bullet,\bullet}$ as 
follows. For every $\alpha \in \RR$ and $c \in \Delta$ we associate to a
section~$s$ of~$\cEE$ in a neighborhood of~$c$ the sections 
$s_\alpha = t^{\lfloor \alpha \rfloor} s$ resp.\
$s_\alpha =t^{\lfloor \alpha \rfloor +1}\, s$ 
depending on whether germ of~$s$ in the stalk of~$\cEE$ 
belongs to $V^{\geq \{\alpha\}}$ or not. We define $\cEE_{c,\alpha}$ to be the
subspace generated by all the section~$s_\alpha$ obtained in this way.

\subsection{The Deligne extension} \label{sec:DelExt}

Here we recall the construction of Deligne's extension of the bundle $\cVV_C$
with flat connection to a holomorphic vector bundle $\cVV$ on $\ol{C}$ with a logarithmic 
connection. The hypothesis on the non-expanding cusp monodromies
implies that $\cVV$
has a canonical\footnote{The choice of the interval $[0,1)$ is an artificial
choice of a unit interval in~$\RR$ and so Deligne (\cite{delequadiff}) calls this extension
{\em quasi-canonical}.} parabolic structure, as we now explain.
\par
To construct the Deligne extension of $\cVV_C$ we use a small disc $D$
centered around the point $c \in \Delta$ with coordinate $q$. We choose a
base point $c_0 \in D \setminus\{c\}$, the conjugation by moving the base
point will not affect the extension. We let $T = T(\gamma) \in \GL(V_0)$ be the monodromy 
of the the flat bundle $\VV$ along a loop $\gamma$ once around $c$, where
$V_0 = (\cVV_C)_{c_0}$ is the fiber over the base point $c_0$. For
every $\alpha \in [0,1)$ we can define
\be
W_\alpha \= \{ v \in V_0: (T - \zeta_\alpha)^{r}v = 0\}, \quad \text{where} \quad
\zeta_\alpha \=e^{2\pi i \alpha}  \quad \text{and} \quad r = {\rm rk}(\VV)\,.
\ee
These vector spaces are zero for all but finitely many $\alpha_i \in [0,1)$.
Finally, we define
\bes T_\alpha = \zeta_{\alpha}^{-1} T|_{W_\alpha} \quad \text{and} \quad
N_\alpha = \log T_\alpha\,,
\ees
since $T_\alpha$ is unipotent.
\par 
Let $q: \HH \to D^*$, $q(z) =e^{2\pi i z}$ be the 
covering of $D^* = D \sms \{c\}$. Choose a basis $v_1,\ldots,v_r$ of $V_0$ adapted to the
direct sum decomposition $V_0 = \oplus_\alpha W_\alpha$. Since $\HH$ is simply connected, 
we may view the $v_i$ as sections $v_i(z)$ of $q^* (\VV_C|_{D^*})$. If $v_i \in W_\alpha$, 
then we define
\be \label{eq:vivitilde}
\wt v_i(z) \= \exp(2\pi i \alpha z + z N_\alpha) v_i\,.
\ee
These sections are constructed to be equivariant under $z \mapsto z+1$, hence they give
global sections of $\cVV_C(D^*)$. The Deligne extension $\cVV$ of $\cVV_C$ is the 
vector bundle, whose space of sections over $D$ is 
the $\cOO_D$-module spanned by $\wt v_1,\ldots,\wt v_r$. 
\par
This construction naturally gives a parabolic structure on the special fiber 
$V_x = (\cVV_C)_{x}$. We let $V_\alpha$ be the subspace generated by the $\wt v_i$ with 
$v_i \in W_\alpha$ and we let $V^{\geq \alpha} = \oplus_{\beta \geq \alpha} V_\beta$ to obtain a 
filtration $F^\bullet_\gamma$ on $V_x$. 

\subsection{Metric extension, acceptable and admissible metrics} \label{sec:admmetric}

The notion of an admissible metric serves two technical purposes. On one hand 
it should specify the correct metric extension by imposing appropriate growth near 
the cusp while on the other hand giving an integrable flat bundle.  This section
follows the treatment of metric extensions of vector bundles and local systems
given in~\cite{SiConstructing}, Section~10 and ~\cite{SiNoncomp}.
\par 
As preparation for the definition, we first recall the notion of 
{\em metric extension}~$\Xi(\cEE_C)$
of a vector bundle~$\cEE_C$ on $C$. Let $j: C \to \ol{C}$ be the inclusion. Given a 
metric~$h$ on~$\cEE$ we define $\Xi(\cEE_C)$ to be the family of subsheaves 
of $j_*\cEE_C$ indexed by $\alpha \in \RR$ such that sections 
$s(q)$ of $\Xi(\cEE_C)^{\geq \alpha}$ are those holomorphic sections that satisfy 
the following ``growth'' condition. For all $\ve \geq 0$ there exists $C_\ve$ such 
that\footnote{\cite[Section~10]{SiConstructing} has a typo, the
exponent there is erroneously $\alpha + \ve$.}
\be \label{eq:growthbound}
 |s(q)|_h \,\leq\, C_\ve |q|^{\alpha - \ve}\,.
\ee
In general, the metric extension of a vector bundle is
a coherent sheaf, not a vector bundle. We will, however, use metric extensions only when
they are vector bundles, in fact Deligne extensions of local systems, see 
Lemma~\ref{le:hasadmmetric} below.
\par
Following~\cite{SiNoncomp} we say that a smooth metric
$h =\langle \cdot,\cdot \rangle$ on the bundle~$\cEE_C$ on the curve~$C$
(provided with the Poincar\'e metric) is {\em acceptable}, if the curvature
of the metric~$h$ admits locally near every $x \in \Delta$ a bound 
\be \label{def:accept}
 |R_h| \leq f + \frac{C}{|q|^2 |\log(q)|^{2}} \quad \text{with} \quad f \in L^p \quad
\text{for some} \quad p>1\,.
\ee
We also say that $h$ is an acceptable metric on a filtered vector bundle 
$\cEE = \{\cEE_{\bullet,\bullet}\}$ if the metric $h$ is acceptable on $\cEE|_C$ 
and $\cEE = \Xi(\cEE|_C)$.
\par
For integrability purposes we require for admissibility growth rates that are slightly 
more restrictive than~\eqref{eq:growthbound}, but obviously imply this bound.
\par
\begin{Defi}
A smooth metric $h =\langle \cdot,\cdot \rangle$ on the bundle~$\cVV_C$ 
with underlying local system~$\VV$ is called {\em admissible}, if for every cusp
$c \in \Delta$ with local coordinate $q$
\begin{itemize}
\item[i)] the metric extension $\Xi(\cVV_C)$ with respect to~$h$ is isomorphic 
as filtered vector bundle to the Deligne extension~$\cVV$ of~$\cVV_C$, 
\item[ii)] for any $e \in \Xi(\cVV_C)^{\geq \alpha}$ and any $e' \in 
\Xi(\cVV_C)^{\geq \alpha'}$ there is some $n \in \NN$ and $C_1=C_1(e,e') > 0$ 
independent of~$q$ such that 
$$\langle e,e' \rangle \,\leq\, C_1\, |q|^{\alpha+\alpha'} (\log |q|)^{2n}\,,$$
\item[iii)] there is some $n \in \NN$ and $C_2 > 0$ such that a generating section~$e$
of $\det(\cVV)$ has the lower bound
$$ ||e||^2  \,\geq\, C_2\, |q|^{2 \dim_{F^\bullet} \cVV_c} (\log |q|)^{-2n}\,.$$
\item[iv)] and, moreover, if the metric is acceptable.
\end{itemize}
\end{Defi}
\par
\medskip
In our situation, the relevant existence statement is the following lemma, 
that follows from Theorem~4 in~\cite{SiNoncomp}.
\par
\begin{Lemma} \label{le:hasadmmetric}
A local system $\VV$ with non-expanding cusp monodromies
has a metric which is admissible for its Deligne extension~$\cVV$. 
\end{Lemma} 
\par
\par
\begin{proof} It suffices to construct such metrics locally and patch them 
with the help of a partition of unity. On the complement of cusp
neighborhoods we can take any metric. On the cusp neighborhoods it 
suffices to treat each eigenspace for the monodromy separately  
and declare the different eigenspaces to be pairwise orthogonal.
The basis elements $\widetilde{v_i}$ of the $\alpha$-eigenspace of the Deligne 
extension are given the norm $|q|^\alpha$ in the local coordinate $q$ around
the cusp and defined to be pairwise orthogonal. 
This implies that the Deligne-extension is the metric extension and
that the norm bounds ii) and iii) hold. The fact that such a metric
satisfies the curvature bound for being acceptable can be calculated directly, 
see also~\cite{SiNoncomp}, Section~5.
\end{proof}
\par
\medskip
In the proof of the main theorem it will be convenient to pass to exterior
powers. We now provide the necessary background in the case of parabolic bundles.
First note, that if the metric $h$ is acceptable on a bundle $\cEE$, then
the induced metric on any exterior power of $\cEE$ is again acceptable. There
are two natural ways to define its exterior powers as filtered vector bundles.
One is to declare $v_1 \wedge \cdots \wedge v_k$ to lie in $(\wedge^k \cEE)_{c,\alpha}$, 
if and only if $\alpha \leq \sum \alpha_i$ where $\alpha_i$ is maximal with 
$v_i \in \cEE_{c, \alpha_i}$.
The second possibility is to take $\Xi(\wedge^k(\cEE|_C))$. It is obvious from the
definition that $\wedge^k(\cEE)_\alpha \subseteq \Xi(\wedge^k(\cEE|_C))_\alpha$.
It was shown by Simpson (\cite{SiNoncomp}, Proposition~3.1,  using the calculations 
leading to \cite{SiConstructing}, Corollary~10.4, in particular the Remark on p.~911) 
that accessibility of~$h$ implies that the converse inequality also holds, i.e.\
\bes \wedge^k(\Xi(\cEE|_C)) \= \Xi(\wedge^k(\cEE|_C) \ees
and so both definitions of the exterior power agree.
\par
\begin{Prop} \label{prop:degwedgepower}
If~$\cEE$ is a vector bundle of rank $k$ then $\deg_{\rm par} \cEE = 
\deg_{\rm par} (\wedge^k \cEE)$.
Moreover, any acceptable metric $h$ computes the parabolic degree of $\cEE$, i.e.\ 
\bes
\deg_{\rm par}(\cEE, F^\bullet) \= \frac{1}{2\pi i} \int_C \partial \ol{\partial} 
\log (\det h_{ij}) \,,
\ees
where $h_{ij} = \langle e_i, e_j \rangle$ are the coefficients of the acceptable metric.
\end{Prop}
\par
\begin{proof} 
The first statement is a direct consequent of the first definition of the exterior power.
\par
By the first statement and since $\det h_{ij}$ is the coefficient of the induced 
acceptable metric on the the $k$-th power (obvious from the second definition), 
we may suppose that $\cEE$ is a line bundle.
For any choice of a generating local section $e =e(q)$ near a point $c \in \Delta$
and a smooth metric $h_\ve$ that agree with $h$ outside $\ve$-neighborhoods
of the cusps we have (see e.g.\ \cite[page 60-61]{Kaw} for details)
\ba
\deg(\cEE) &\=  \frac{1}{2\pi i} \int_C \partial \ol{\partial} \log h_\ve \\
 &\=  \frac{1}{2\pi i} \Bigl( \int_C \partial \ol{\partial} \log h
\+ \sum_{c \in \Delta} \lim_{\ve \to 0} \int_{|q|=\ve} \ol{\partial} \log 
\langle e(q), e(q) \rangle \Bigr) \\
 &\=  \frac{1}{2\pi i} \int_C \partial \ol{\partial} \log h \,-\, \sum_{c \in \Delta} 
\dim_{F_c^\bullet} \cEE_c
\ea
and this proves the claim.
\end{proof}

\subsection{Proof of the integrability statements} \label{sec:pfintegrab}

It is obvious that in order to prove Theorem~\ref{thm:admisintegrable}
and Proposition~\ref{prop:constantmetric} it suffices to prove the following
two lemmas. We use the cocycle language for the flat bundle, as
introduced in the appendix.
\par

\begin{Lemma}
The cocycle~$A$ induced by the geodesic flow on a hyperbolic surface with cusps
on a normed flat bundle with non-expanding cusp monodromies
is integrable for a constant norm.
\end{Lemma}
\begin{proof}
We have to estimate the growth of the norm over a geodesic segment of
length one. Consider a complement $C_\varepsilon$ to a neighborhood
of cusps. If the starting point is located in $C_\varepsilon$ then
the geodesic segment of unit length starting at this point can cross
the boundary of the Dirichlet domain only finite number of times where the bound
is uniform for all starting points. Thus,
the growth of the constant norm is uniformly bounded for such
segment (where the bound depends on the flat bundle,
on the Dirichlet domain and on the choice of $\varepsilon$).

It remains to estimate the growth of the norm for a geodesic segment
of unit length starting in a small neighborhood of a cusp. Since the
boundary of the Dirichlet domain near a cusp is represented by a
geodesic ray going straight to the cusp, we have to count how many
times such a geodesic segment could turn around the cusp. Consider
standard coordinates in the neighborhood of the cusp, namely take a
half-strip $-\frac{1}{2}\le x\le \frac{1}{2}$, $y\ge y_0\gg 1$ in the
upper-half plane with coordinates $z=x+iy$ and with hyperbolic
metric $g$ of constant negative curvature $-4$,
\begin{equation}
\label{eq:g:hyp}
g\=\frac{|dz |^2}{4 (\operatorname{Im} z)^2} \= \frac{dx^2+dy^2}{4y^2}\,.
\end{equation}
The upper
bound of the number of turns around the cusp of a geodesic segment of unit 
length starting at a point $x+iy$ with $y\ge y_0\gg 1$
is given by the path which
first goes straight to the cusp for time $1$
and then  follows the closed horocycle
around the cusp for time $1$.
\par
The first segment starts at a point $x+iy$ and goes vertically up to
the point $x+ie^2y$.  The hyperbolic length of the
closed horocycle around the cusp located at the height $y=e^2y$ is
$\frac{1}{2e^2y}$, so the path following the closed horocycle for time~$1$ makes
at most $2e^2 y+1$ turns around the cusp. The condition on
non-expanding cusp monodromy implies that the norm of a constant
vector transported $N$ times around the cusp grows linearly in $N$.
Hence for $x+iy \in C_\varepsilon$ the growth of the constant norm
is bounded by 
$$
\max_{t\in[-1,1]}\log^+\!\|A(x+iy,t)\|\,<\, c_1 \log y +c_2
$$
for some constants $c_1,c_2\in\mathbb{R}$ depending on the flat bundle.
Clearly,
$$
\int_{-\frac{1}{2}}^{\frac{1}{2}} dx \int_{y_0}^{+\infty}
\left(c_1\log y + c_2\right)\frac{dy}{4y^2} \,<\, +\infty\,,
$$
and the integrability of the cocycle for the constant norm follows. 
\end{proof}
\par
The notion of \textit{equivalent norms} for integrable cocycles is
definitely known, see, for example, the corresponding
remark in~\cite{Ruelle}. However, since this notion is important in the 
context of this paper, for the sake of completeness
we collect all necessary details in the appendix.
\par
\begin{Lemma}
A constant norm and an admissible norm~$h$ are $L^1(\mu)$-equivalent.
\end{Lemma}
\begin{proof}
Consider standard coordinates in the neighborhood of the cusp, namely
take a half-strip $-\frac{1}{2}\le x\le \frac{1}{2}$, $y\ge y_0$ in
the upper-half plane with coordinates $z=x+iy$ and hyperbolic
metric $g$ as in~\eqref{eq:g:hyp}. Consider a geodesic ray
$\{x_0 +iy\,|\, y\ge y_0\}$ going straight to the cusp. Consider a
section $\vec v_{x_0+iy)}$ of the flat bundle over the geodesic ray constant with
respect to the flat connection.
The coordinate $q$ in a punctured disk around the cusp
is related to our coordinate~$z$ as above as
$$
2\pi i z \=\log q \quad \text{i.e.} \quad \log|q| \= -2\pi y\,.
$$
By condition i) of admissibility the flat section $\vec v_{x_0 +iy}$ can 
be expressed on the half-strip as a linear combination of either the basis 
elements~$\tilde{v_i}$ or the basis elements~$v_i$ introduced along with
the definition of the Deligne extension. By  condition ii) of admissibility
the sections $\exp(-2\pi i\alpha z)\tilde{v_i}$ is bounded above by
$C \log|q|^{2n}$ for some~$C$ and~$n$. By the conversion~\eqref{eq:vivitilde} 
the flat sections~$v_i$ and hence also $\|\vec v_{x_0 +iy} \|_{\rm adm}$
is bounded above by $C \log|q|^{2n'}$ (for an appropriate choice of the constant, 
depending on the monodromies $N_\alpha$).
The lower bound for the determinant given by 
condition iii) of admissibility and Cramer's rule imply that the norm
of such a non-zero flat section is bounded below by $C' \log|q|^{-2n''}$. 
Thus the ratio of a constant norm and an
admissible norm is uniformly bounded in the complement of
neighborhoods of the cusps and has the form
$$
\max_{\vec v\in\mathcal{V}_{x_0 + iy}\setminus \vec 0}
\left|\log\frac{\|\vec v\|_{\rm adm}} {\|\vec v\|_{\rm const}}\right|
=
\max_{\vec v\in\mathcal{V}_{x_0 + iy}\setminus \vec 0}
\left|\log\frac{\|\vec v\|_{\rm const}}{\|\vec v\|_{\rm adm }}\right|
\le
K \log y\,.
$$
in the local coordinates in the
neighborhood of a cusp. The integral
$$
\int_{-\frac{1}{2}}^{\frac{1}{2}} dx \int_{y_0}^{+\infty}
\log y\, 
\frac{dy}{4y^2}
$$
converges, so the constant norm and the admissible norm are
$L^1$-equivalent, and Theorem~\ref{th:equivalent:norms} implies that
the cocycle corresponding to the admissible norm is integrable and
defines the same Lyapunov exponents as the one corresponding to the
constant norm. 
\end{proof}

\subsection{Necessity of the non-expanding condition} 
\label{sec:nec}

We remark that if there exists a cusp $c_0$ of $C$ such that 
at least one of the eigenvalues of the monodromy around this cusp
has absolute value different from one, then the flat bundle~$\VV$
is not integrable with respect to the constant norm. 
\par
\begin{proof}
Suppose that the starting point $p=x+iy$ of a geodesic segment is
located sufficiently high in the cusp, that is $y\ge y_0\gg 1$ in
coordinates~\eqref{eq:g:hyp}. We consider the geodesics launched from
$p$ in direction $\xi$ from the subset
$[\frac{\pi}{6},\frac{\pi}{3}]\cup[\frac{2\pi}{3},\frac{5\pi}{6}]\subset[0,2\pi]$.
The direction is chosen to make the geodesic spiral toward the cusp
so that its $y$ coordinate still grows at least for some uniform
starting time $\varepsilon(y_0)>0$ depending only on parameter $y_0$.
We have chosen our geodesic to go not too steep to the cusp. The
angle between the geodesic $\gamma_{t}(p,\xi)$ as above and the
vertical direction would only grow for $t\in[0,\varepsilon]$, so the
horizontal projection of the geodesic has speed at least
$\frac{1}{2}$ for the entire interval of time $[0,\varepsilon]$. Since
the cusp at height $y$ has width $\frac{1}{2y}$ and for the
time $\varepsilon$ the geodesic does not get below the initial height
$y$, we conclude that in the interval of time $[0,\varepsilon]$ it
makes at least $y\varepsilon-1$ turns around the cusp.

Suppose that there is an eigenvalue of the mondromy around the cusp
such that its absolute value is different from one. Let
$a\neq 0$ be the logarithm of this absolute value.
The calculation above shows that for any geodesic as above we have
$$
\sup_{t\in[-\varepsilon,\varepsilon]}\log^+\!\|A(\gamma_{t}(p,\xi))\|_{\rm const}\ge (y\varepsilon-1)\cdot a\,.
$$
The subset of starting directions allowed above
has $1/6$ of the measure of all unit circle.
Since the integral of the function $y$ is diverging
with respect to the measure~\eqref{eq:g:hyp} near the cusp, this implies that
the integral
$$
\int_{T^1 C} \ \ \sup_{t\in[-1,1]}\log^+\!\|A(\gamma_{t}(p,\xi))\|_{\rm const}\, d\mu(x)
$$
is diverging and, hence, that the flat bundle $\VV$ is not integrable.
\end{proof}

\section{Existence of Lyapunov exponents for variations 
of Hodge structures} \label{sec:ExLVHS}

In this section we show that the Hodge metric for families of varieties
or more generally for a real variation of Hodge structures satisfies
the admissibility assumption of Section~\ref{sec:admmetric}. 
For a variation of Hodge structures we sketch, moreover, that there are 
{\em uniform} bounds for 
the Lyapunov exponents depending only on rank and weight of the VHS. An interesting
open problem is to prove sharp estimates and interpret the families that
reach the upper bounds geometrically.
\par
\smallskip
We recall the definition of real and complex variations of Hodge structures (VHS), 
also to introduce the Hodge metric. A {\em  $\CC$-VHS} on the curve~$C$ consists 
of a complex local system~$\VV_\CC$ with connection $\nabla$
and a decomposition of the Deligne extension $\cVV = \bigoplus_{p\in\ZZ} \cEE^{p}$ into
$C^{\infty}$-bundles, such that 
\begin{itemize}
\item[i)] $\cFF^p := \bigoplus_{i\geq p} \cEE^{i}$ are holomorphic subbundles
and $\ol{\cFF^p} := \bigoplus_{i\leq p} E^{p}$ are antiholomorphic
subbundles for every $p \in \ZZ$ and
\item[ii)] the connection shifts the grading by at most one, i.e.\ 
$\nabla(\cFF^p) \subset \Omega^1_C \otimes \cFF^{p-1}$ and 
$\nabla(\ol{\cFF^p}) \subset \Omega^1_C \otimes \ol{\cFF^{p+1}}$.
\end{itemize}
\par
\smallskip
To define the notion of $\RR$-VHS we first recall that for a 
\emph{real Hodge structure of weight $\ell$} on $W$, we require a decomposition 
$W\otimes_\RR \CC = \oplus_{p=0}^\ell W^{p,\ell-p}$, such 
that $\ol{W^{p,q}}= W^{q,p}$. An \emph{$\RR$-VHS of weight $\ell$} over the base $C$ 
consists of a $\RR$-local system $\VV$  and a filtration
$$ 0 = \cFF_0 \subset \cFF_1 \subset \cFF_2 \cdots \cFF_{\ell-1} \subset \cFF_\ell \subset 
\cVV$$
on the Deligne extension of $\VV$  with the property that the bundles
$\cHH^{p,q} = \cFF^p \cap \ol{\cFF^q}$ fiberwise define an $\RR$-Hodge structure.
\par
An $\RR$-VHS $\WW$ is {\em polarized}, if there exists a non-degenerate, 
locally constant bilinear form $Q(\cdot,\cdot)$ on $\WW$, skew for $\ell$ odd
and symmetric for $\ell$ even, such 
that $Q(\cHH^{p,q},\cHH^{r,s}) = 0$, unless $p=s$ and $q=r$, and such that 
$i^{p-q} Q(v,\ol{v}) >0$ for every non-zero $v \in \cHH^{p,q}$. Consequently
if we define an endomorphism $S$ of $\VV \otimes_\RR \cOO_C$ by
$S(v) = i^{p-q}v$ for $v \in \cHH^{p,q}$, then the {\em Hodge scalar product}
$h(v,w) =Q(Sv,\ol{w})$ is positive definite. We let $||\cdot||_h$ 
be the associated {\em Hodge norm} of $\VV_\CC$. It is obtained by interpreting $\VV$ 
as the direct sum of the smooth subbundles $\cHH^{p,q}$ and by using the 
positive definite metric on each of them.
\par
For any family of projective varieties $f: X \to C$ the $\ell$-th cohomology
gives a polarized $\RR$-VHS of weight $\ell$ in this sense. 
\par
Note that by a theorem of Borel (see e.g.\ \cite{schmid73} Lemma~4.5) the 
non-expanding cusp monodromy
hypothesis holds. If the local system underlying the VHS has a $\ZZ$-structure
(or arises as a direct summand of the cohomology of a family of varieties)
then the monodromies around the cusps are moreover {\em quasi-unipotent}.
\par
\begin{Prop} \label{prop:HodgeAdm}
The Hodge metric on $\VV$ is admissible.
\end{Prop}
\par
\begin{proof} The corresponding estimates were first derived by Schmid 
(\cite{schmid73}).
They are restated in \cite{Pet84}, see Proposition~2.2.1 for the growth rates
and Example~3.2 for how to derive the curvature estimate for acceptability.
\end{proof}
\par
The following result gives a second proof of integrability in this case. Recall that
 $G_t$ denotes the lift of the geodesic flow $g_t$.
\par
\begin{Prop} \label{prop:LyapCoareUpper}
For a VHS the function $x\mapsto \sup_{t\in [0,1]} \log^+ \|G_t\|_x$ is bounded
by a constants depending on the rank and the weight only. Consequently, 
the Lyapunov exponents of a VHS are bounded by a constants depending on the 
rank and the weight only.
\end{Prop}
\par
We make no attempt here to make the estimate precise, since the bound from 
the estimate below is very rough.
\par
\begin{proof}
Let $D$ be the period domain for polarized weight $\ell$ Hodge structures 
with dimensions  of the filtration pieces as given by $\VV$. In general, 
$D$ is not a symmetric domain but just a homogeneous space. The tangent bundle
to $D$ has the so-called horizontal subbundle $T_h$ with two properties.
First, by Griffiths transversality the tangent vectors to the period map
$p: \HH \to D$ for $\VV$ lie in $T_h \subset T_D$. Second, the 
holomorphic sectional curvature of directions in $T_h$ is negative and
bounded away from zero (\cite[Theorem~9.1]{GriSc}), \cite[Chapter~13]{CMSP}), 
say by $K$. This contractivity along the horizontal 
distribution implies the integrability as we now elaborate.
\par
To provide a universal bound it it suffices to bound 
for $\tfrac{\partial}{\partial t} \log h(v(t),v(t))\, |_{t=0} $, 
where $v(t)$ is the parallel transport of a unit norm vector $v$ along $g_t$.
We decompose $v(t) = \sum v^{p,q}(t)$ into its Hodge components and
write $\sigma = \sum \sigma_p$ for the graded pieces $\sigma_p: 
\cHH_p \to \cHH_{p+1}$ of the Gauss-Manin connection contracted against a
unit tangent vector at $t=0$ in the direction of $g_t$. Expanding into
components, we obtain
\ba & \phantom{\,\leq \,2} \frac{\partial}{\partial t} \log h(v(t),v(t)) |_{t=0} 
 \= \frac{\tfrac{\partial}{\partial t} h(v(t),v(t))}{h(v,v)} \\
& \,\leq\, 2\, \frac{\sum_{p=0}^{\ell-1} h(\sigma_p (v^{p,\ell-p}), v^{p+1,\ell-p-1}) + 
h(\sigma_{p+1}^\dagger (v^{p+1,\ell-p-1}), v^{p,\ell-p})}
{\sum_{p=0}^\ell h(v^{p,\ell-p},v^{p,\ell-p})} \, ,
\ea
where $\sigma^\dagger: \cHH_{p+1} \to \cHH_{p}$ is the adjoint of $\sigma$. From
this expression it is obvious that it suffices to bound from above the operator norms
of all the maps $\sigma_p$, hence of $\sigma$. Since $D$ is homogeneous and
finite dimensional, any two norms are comparable, so we may as well bound
the euclidean norm $\sigma$. But since $\sigma$ is just the derivative
of the period map $p$, we can now invoke the Ahlfors-Lemma in the
version of \cite[Theorem~2]{Roy} to obtain the 
bound $||\sigma||_2 \leq ||dp||_2 \leq \sqrt{|k|}/\sqrt{|K|}$, where
$k=-4$ is the curvature of $\HH$ in the convention we use. 
\end{proof}
\par

\section{The bad locus and the main estimate} \label{sec:badandmain}

Suppose  that  we are given a $\CC$-local system $\VV$ of
rank~$r$ over a curve~$C$ with 
non-expanding cusp monodromies. 
Let  $\Delta$  be  the  set  of  boundary  points  of $C$ and recall
that by assumption $\chi(C)=-\deg\Omega^1_C(\Delta)<0$.
Denote by $\lambda_1\ge\dots\ge\lambda_r$  the Lyapunov exponents of $\VV$
with respect to the norm $\|\cdot\| = \|\cdot\|_h$ stemming from 
an admissible metric~$h$ as given by Theorem~\ref{thm:admisintegrable}.
Note that the metric on~$\VV$  naturally equips the dual bundle~$\VV^\vee$ 
with an admissible 
metric $\|\cdot\|_\vee$ defined by $\|u\|_\vee = \sup_{v \neq 0} |u(v)|/\|v\|_h$, 
which is admissible as well (\cite{SiNoncomp}, Theorem~4), and can be used 
to compute the Lyapunov exponents of~$\VV^\vee$.
\par
In this section we prove a conjecture of Fei~Yu~\cite{Fei}, or more
precisely, a generalization from the context of VHS to the case of
local systems with non-expanding cusp monodromies.
\par
\begin{Thm} \label{thm:main_estimate}
If $\cEE \subset \cVV$ is a holomorphic parabolic subbundle of
rank $k$ of the Deligne extension $\cVV$ of $\VV \otimes_\CC \cOO_C$, then 
\be
\label{eq:main:ineqver2}
\sum_{i=1}^k       \lambda_i       \,\geq\,
\frac{2\deg_{\rm par}(\cEE)}{\deg\Omega^1_{\ol{C}}(\Delta)} \=
\frac{2\deg_{\rm par}(\cEE)}{2g(\ol{C})-2+|\Delta|}
\,. \ee
\end{Thm}
\par
We  do not assume that the flat bundle $\VV$ is irreducible. Clearly,
the  theorem is applicable to every irreducible summands of~$\VV$, so 
if~$\VV$  is reducible,  we can decompose $\VV$ into a direct sum of irreducible
summands and obtain finer estimates by applying the theorem
individually to each irreducible summand.
\par
\medskip
\par
The condition ``parabolic subbundle'' refers to the parabolic structure on $\cVV$
introduced in Section~\ref{sec:DelExt}. This condition is void for unipotent
monodromies. For Teichm\"uller curves one can always restrict to this case.
In fact, we can in this case (tacitly) replace $C$ by a finite unramified
covering such that the local monodromies around the cusps in $\Delta$
are unipotent. This is always possible, since in general local monodromies
are quasi-unipotent and since $\pi_1(C)$ is finitely generated and free
if $C$ is not compact. This base change does not modify Lyapunov exponents
and multiplies numerator and denominator of the right hand side 
of~\eqref{thm:main_estimate} by the degree of the covering.
\par
\medskip
We prepare for the proof with three reduction steps. First note that
the Lyapunov spectrum of~$\VV$ is symmetric with respect to zero, i.e.\ 
$-\lambda_{r+1-\ell} = \lambda_\ell$ for any~$\ell$. This follows since
the geodesic flow
in negative time has on the one hand the negative of the Lyapunov spectrum
for every flow and on the other hand (due to the ${\rm SL}_2(\RR)$-action on $\HH$) 
the flow in positive and negative time are conjugated and have consequently
the same Lyapunov spectrum.
\par
Second, we remark that $\lambda_i(\VV) = - \lambda_{r+1-i}(\VV^\vee)$. 
Combining these two observations it suffices to prove that 
\be
\label{eq:main:ineqproven}
\sum_{i=1}^k       \lambda_i(\VV^\vee)       \,\geq\,
\frac{2\deg_{\rm par}(\cEE)}{\deg\Omega^1_{\ol{C}}(\Delta)} \,. \ee
\par
Moreover, we remark that it suffices to prove the 
theorem for the case when  $\cEE$  is a \textit{line} bundle, i.e.\ to treat the
case $k=1$. In fact, given a parabolic subbundle $\cEE \subseteq \cVV$  as in 
statement of the theorem, the exterior power $\cLL = \wedge^k \cEE$ is a
parabolic subbundle of $\wedge^k  \cVV$ (obvious from the second definition
of exterior powers as defined in Section~\ref{sec:admmetric}) and $\wedge^k  \cVV$ 
is the Deligne extension of $\wedge^k \VV \otimes  \cOO_C$. (This also
follows from the first defining property of an \adm\ metric and the
compatibility of the metric with taking exterior powers.)
Since $\deg_{\rm par}(\cEE) = \deg_{\rm par}(\cLL)$ by 
Proposition~\ref{prop:degwedgepower} and since the top Lyapunov exponent of 
$\wedge^k \VV$ is just $\sum_{i=1}^k \lambda_i$ the claimed reduction to $k=1$ 
follows.
\par
Mimicking the idea of~\cite{kontsevich} we define an auxiliary norm on the
dual bundle~$\cVV^\vee$ by defining for any point $u$ in the total space 
of~$\cVV^\vee$
\be
\label{eq:seminorm:of:L}
||u||_{\cEE} \,:=\,
\frac{|\omega_c(u)|} {\sqrt{|h(\omega_c,\omega_c)|}}
\= \frac{|\omega_c(u)|} {\|\omega_c\|_h}\, ,
\ee
where $\omega_c$ is a nonzero element of the fiber $\cEE_c$ over
the point $c$ in $C$.
This seminorm is well-defined, i.e.\ it does not depend on the
choice  of the nonzero vector $\omega_c$ in $\cEE_c$, since numerator
and denominator are homogeneous of the same degree in $\omega_c$.
\par
The  difference  with the standard case of weight  one   
(see~\cite{kontsevich}, \cite{forni02}, \cite{bouwmoel}, 
\cite{cyclicekz}) is that the numerator can indeed become zero. We call
the locus where the numerator in~\eqref{eq:seminorm:of:L} vanishes
the {\em bad locus} with respect to $\cEE$, that is, we define
$$ T^\bad \= \{(c,u)\, \colon \omega_c(u)=0\} $$
as a subset of the total space of the bundle~$\cVV_C$ over~$C$.
\par
Since Lyapunov exponents are defined by parallel transport, we really need
a definition of the bad locus that records all translates of a given vector.
Let $p:\HH\to C$ denote the universal cover. The flat structure on $\cVV_C^\vee$
provides a trivialization of $\p^* \cVV_C^\vee$. Using the parallel transport
of section given by this trivialization we define for $u \in \cVV_C^\vee$
the \textit{bad locus of $u$}
as
\be \label{eq:Tbad}
T^\bad(u) \= \{ z \in \HH\ \colon \omega_z(u) =0\}\, . 
\ee
Here  $\omega_z$  is  a  generator in the fiber $p^\ast\cEE_z$ of the
induced bundle $p^\ast\cEE$ over $\HH$. In other words, $T^\bad(u)$ is
the  set  of  points  $z$ in $\HH$ for which the fiber $p^\ast\cEE_z$
of the line bundle $p^\ast\cEE$ gets inside the hyperplane $\Ann u$.
\par
\begin{Lemma}
\label{lm:discrete:subset}
For every $c \in C$ there is a countable union of hyperplanes $H$ in $\cVV_c^\vee$
such that for $u \in \cVV_c^\vee \setminus H$ the bad locus
$T^\bad(u)$ is a discrete subset of $\HH$.
\end{Lemma}
\par
\begin{proof}
Since $\cEE$ is a holomorphic subbundle, locally $T^\bad(u)$ is given
as  the vanishing locus of a holomorphic function. Thus, for any given $u$ 
the locus $T^\bad(u)$ is either discrete in $\HH$ or equal to $\HH$. The second
possibility might only occur on the hyperplane $\Ann~\omega_c$. (The countable union 
results from the choice of a $p$-preimage.)
\end{proof}
\par
One can prove in fact that, if the flat bundle $\VV$ is irreducible over $C$, 
the subbundle of $\cVV^\vee$ given by those $u$ for which $T^\bad(u)=\HH$ is actually 
the zero bundle. Next, we compare the admissible metric and the 
$\|\cdot\|_{\cEE}$-seminorm.
\par
\begin{Lemma}
\label{lm:norm:over:seminorm}
For any point $c$ of the curve $C$ and for any $u$
in the fiber $\cVV_c^\vee$ over~$c$ 
\begin{equation}
\label{eq:norm:over:seminorm}
\cfrac{\|u\|_\vee}{\|u\|_\cEE} \,\ge\, 1\,.
\end{equation}
\end{Lemma}
\par
\begin{proof} The definition of the norm on the dual bundle implies
 $ |\omega_c(u)|\le \|\omega_c\|_h\cdot \|u\|_\vee$, implying the claim.
\end{proof}
\par
\begin{proof}[Proof of Theorem~\ref{thm:main_estimate}]
Pull back the flat bundle and the holomorphic linear subbundle
$\cEE$ to the universal cover $\HH$ over $C$.
For any $z\in \HH$ and for almost any $u$ in the fiber $\VV_z^\vee$ over $z$
one can express the Lyapunov exponent $\lambda_1(\VV^\vee)$
(see~\cite{eskozo}, \S 3.2) as
$$
\lambda_1(\VV^\vee) \= \lim_{T \to \infty} \frac{1}{T} \frac{1}{2\pi} \int_0^{2\pi}
\log ||g_T r_\theta u||_\vee d\theta\,.
$$
\par
Now we replace the admissible norm $\|\cdot\|_\vee$
by the seminorm $\|\cdot\|_\cEE$. Lemma~\ref{lm:norm:over:seminorm} implies 
the inequality 
\begin{equation}
\label{eq:intermidiate}
\lim_{T \to \infty} \frac{1}{T} \frac{1}{2\pi} \int_0^{2\pi}
\log ||g_T r_\theta u||_\vee d\theta
\,\ge\,
\lim_{T \to \infty} \frac{1}{T} \frac{1}{2\pi} \int_0^{2\pi}
\log ||g_T r_\theta u||_\cEE d\theta\,.
\end{equation}
   %
A priori, the limit in the right-hand-side might be equal to
$-\infty$. As an outline for the remaining proof, we want to run the standard 
argument (compare e.g. \cite{forni02},
or \cite{eskozo} \S  3.2--3.3, see \cite{kappesmoe}, proof of Theorem~3.3 with $n=1$ 
and $\kappa = 1/2$, for details allowing to trace the origin of the normalizing 
factor~$2$ in the numerator
given our curvature conventions) for computing Lyapunov exponents  in
terms of degree of holomorphic subbundle.  We use that from the definition of the 
seminorm $\|g_t r_\theta L\|_\cEE$ in~~\eqref{eq:seminorm:of:L} we get
$$
\log ||g_t r_\theta u||_\cEE \=
\log(|\omega_{g_t r_\theta z}(g_t r_\theta u)|) -
\log(\|\omega_{g_t r_\theta z}\|_h)\,.
$$ 
In contrast to the classical case we need to consider the Laplacian of the 
first summand on the right hand side. Away from $T^{\bad}(g_tr_\theta u)$ the
argument of the logarithm is a non-zero holomorphic functions and since $\Delta_{\hyp}$ 
is proportional to $\partial \overline{\partial}$ this contribution vanishes.
Near a bad point, the local contribution is the integral of $\Delta_{\hyp} \log(|z|^n)$
for some positive~$n$, hence positive. Altogether, we argued (by integrating
over the hyperbolic disc~$D(u)$ around the base point of ~$u$ swept out 
by $g_t r_\theta$) that, for almost every~$u$,
\be \label{eq:logcomparison}
\int_{D(u)} \Delta_\hyp \log ||u_z||_\cEE dg_{\rm hyp} (z)
\,\geq \, -\int_{D(u)}
\Delta_\hyp \log(\|\omega_{z} \|_h) dg_{\rm hyp} (z) \,.
\ee
Here $u_z$ is the parallel transport of~$u$ to the point $z \in D(u)$.
This inequaliy will imply that $\lambda_1(\VV^\vee)$ is greater
or equal to the parabolic degree of~$\cEE$, suitably normalized.
\par
To be self-contained, we reproduce this computation in detail.
Let~$D_t$ be the hyperbolic disc of radius $t$ and
$\Delta_{\hyp}$ be the Laplacian for the hyperbolic metric~$g_\hyp$
on~$D_t$.  In the following chain of (in)equalities, we first apply
an extra averaging over the unit tangent bundle~$T^1C$. Next, we
apply  a version of Green's formula (\cite[Lemma~3.1]{forni02} or
\cite[Lemma 3.6]{eskozo}) for the disc $D_t(u)$ centered around the base 
point of~$u$ of hyperbolic radius~$t$. The subsequent inequality follows 
from~\eqref{eq:logcomparison}. 
Then we exchange the $T$-limit and the $C$-integration,
justified by dominated convergence due to the accessibility of the metric~$h$. 
The resulting double integration over $C$ and $D_t({u})$ both 
just shift the base point
and can be subsumed into a single integration. To pass to the next
line, we use that the integrand no longer depends on~$T$ and
interchange the order of integration again. Finally we pass from 
$\Delta_{\rm hyp}$ to $\partial \ol{\partial}$.
\begin{align*}
\vol(C) \lambda_1(\VV^\vee)
& \geq \int_{T^1C} \lim_{T \to \infty} \frac{1}{T} \frac{1}{2\pi} \int_0^{2\pi}
\log ||g_T r_\theta u||_\cEE d\theta \dd \mu_{T^1C}(u)  \\
     &=  \int_{T^1C} \lim_{T \to \infty} \frac{1}{T} \frac{1}{2\pi} \int_{0}^T
	\frac{\dd}{\dd t} \int_0^{2\pi} \log \|g_t r_\theta u\|_\cEE
	\dd \theta \dd t \dd \mu_{T^1C}(u)
\\
     & =  \int_{T^1C} \lim_{T \to \infty} \frac{1}{T}\int_{0}^T
	\frac{\tanh(t)}{2\vol(D_t)} \int_{D_t({u})} \!\!\!\!\!\!
\Delta_{\rm hyp} \log \|u_z\|_\cEE
	\dd g_{\hyp}(z) \dd t \dd \mu_{T^1C}(u)
\\
     & \geq  \int_{C} \lim_{T \to \infty} \frac{1}{T}\int_{0}^T
     \frac{\tanh(t)}{2\vol(D_t)} \int_{D_t({u})} \!\!\!\!\!\!\!\!\! -\Delta_{\rm hyp}
\log \|\omega_{z}\|_h  \dd g_{\hyp}(z) \dd t \dd g_\hyp
\\
     & =  \lim_{T \to \infty} \frac{1}{T}\int_{0}^T
     {\tanh(t)} \dd t
     \,\,
     \int_{C}  -\frac12 \Delta_{\hyp} \log \|\omega_z\|_h \dd g_\hyp(z)
\\
     & =  -\frac12 \int_{C}  \Delta_{\hyp} \log \|\omega_z\|_h\, \dd g_\hyp(z)
      \=  -\frac14 \int_{C}  \Delta_{\hyp} \log |\det h_{ij}|\, \dd g_\hyp(z)
\\
     & = -\frac14 \int_{C}4
        \frac{\partial^2}{\partial z \partial\bar z}
        \log|\det h_{ij}|\,
\frac{i}{2}\,dz\wedge d\bar z      \= \frac{1}{2i}
\int_{C}\partial \ol{\partial}
\log|\det h_{ij}|
\\
     & \= \pi \deg_{\rm par}(\Xi_h(\cEE|_C)) \, \geq\, \pi\deg_{\rm par}(\cEE)\,,
\end{align*}
where the last inequality is justified as follows. 
By the hypothesis $\cEE \subset \cVV$, the hypothesis $\cVV = \Xi_h(\cVV|_C)$
on the metric~$h$ and the definition of a parabolic subbundle,
the metric~$h$ is acceptable for $\Xi_h(\cEE|_C)$, and hence $\Xi_h(\cEE|_C)$
contains $\cEE$ as parabolic subbundle. The degree is decreasing in passing
to subbundles.
(In fact the last inequality would even be an equality by 
Proposition~\ref{prop:degwedgepower}
if the metric~$h$ restricted from~$\cVV$ to $\cEE$ was acceptable for~$\cEE$.)
\par
Taking into consideration that the hyperbolic area $\vol(C)$ in the
hyperbolic metric of constant negative curvature $-4$ has the form 
$
\vol(C)=\frac{\pi}{2}\left(2g(\ol C)-2+|\Delta|\right)
$
we obtain the desired inequality.
\end{proof}
\par
\begin{Rem} \label{rem:curvature}
{\rm
The  normalization  of the constant negative curvature on the Riemann
surface  $C$  to  $-4$ is a matter of pure convention coming, partly,
from  the  tradition  to associate Teichm\"uller geodesic flow to the
action           of           the          1-parameter          group
$\left(\begin{smallmatrix}e^t&0\\0&e^{-t}\end{smallmatrix}\right)$,  
and     to    have
$\lambda_1=1$  for the top Lyapunov exponent of the Hodge bundle over
the  Teichm\"uller geodesic flow. The choice of the constant negative
curvature $-1$ would impose time normalization which is twice slower,
so  the  1-parameter  subgroup  corresponding  to  geodesic  time for
curvature                 $-1$                would                be
$\left(\begin{smallmatrix}e^{t/2}&0\\0&e^{-t/2}\end{smallmatrix}\right)$.  
In other words,
the  Lyapunov  exponents  for  the geodesic flow in constant negative
curvature  $-k^2$  are  $k$  times  the  Lyapunov  exponents  for the
geodesic  flow  in  constant  negative curvature $-1$. The hyperbolic
area  of  the  Riemann  surface  in  the  metric of constant negative
curvature  $-k^2$  is  $k^{-2}$ times the hyperbolic area of the same
Riemann  surface  in  the metric of constant negative curvature $-1$.
The latter is equal to $2\pi(2g(\ol{C})-2+|\Delta|)$.}
\end{Rem}

\section{Application: Lyapunov exponents for the 
Hodge bundle over the Teichm\"uller geodesic flow} \label{sec:AppLforTeich}

Here we give applications of the main theorem for the Teichm\"uller
geodesic flow. The first is a comparison of slope polygons and
the second is a contribution towards the large genus asymptotics of
individual Lyapunov exponents. Both results were observed in \cite{Fei}, 
and proved there conditionally to our main theorem. We assume in this section
that the reader is familiar with the stratification of the moduli
space of abelian differentials and with the notion of Teichm\"uller curves, 
see e.g.\ \cite{kz03}, \cite{zorich06}, \cite{moelPCMI}.

\subsection{Two polygons}

The {\em slope} of a  
vector bundle $\cFF$ on a curve 
is defined as $\mu(\cFF) = \deg(\cFF)/\rank(\cFF)$. A bundle
is called {\em semistable} if it contains no subbundle of strictly larger slope.
A filtration
$$ 0 = \cFF_0 \subset \cFF_1 \subset \cFF_2 \cdots \subset \cFF_g = \cFF$$
is called a {\em Harder--Narasimhan filtration} if the successive
quotients $\cFF_{i}/\cFF_{i-1}$ are semi-stable and the slopes are strictly
decreasing, i.e.\ 
$$\mu_i := \mu(\cFF_{i}/\cFF_{i-1}) \, > \, \mu_{i+1} := \mu(\cFF_{i+1}/\cFF_i)\,.$$
The Harder--Narasimhan filtration is the unique filtration with these
properties. Given such a filtration, one can record the numerical data
in a ``Harder--Narasimhan polygon'' with vertices $(rk(\cFF_i),
2\deg(\cFF_i)/|\chi|)$, where $|\chi| = 2g-2 +|\Delta|$. 
\par
Here, we apply these considerations to a Teichm\"uller curve  $C$ and to 
$\cFF = f_* \omega_{\ol{X}/\ol{C}}$, the direct image of the relative dualizing
sheaf of the family of stable curves $f: \ol{X} \to \ol{C}$. This agrees
with the Deligne extension of first filtration piece of the weight one
VHS associated with $f$.
\par
Similarly, one can record the numerical data of the Lyapunov exponents
in a ``Lyapunov polygon'' with vertices 
$(k,\sum_{i=1}^k \lambda_i)$. 
\par
The Harder--Narasimhan polygon and the Lyapunov polygon share the 
endpoints $(0,0)$ and $(g,2 \deg f_* \omega_{\ol{X}/\ol{C}}/|\chi|)$.  
Applying the main theorem to the subbundles in the Harder--Narasimhan filtration
gives immediately the following result, originally conjectured by Fei Yu~(\cite{Fei}).
\par
\begin{Cor} The Lyapunov polygon of a Teichm\"uller curves lies always above the  
Harder--Narasimhan polygon (equality permitted).
\end{Cor}

\subsection{Lyapunov exponents for strata}

So far, we only have been working over curves. From this we can deduce
properties of Lyapunov exponents for strata thanks to a convergence
result in \cite{bew} for individual Lyapunov exponents.
\par
\begin{Thm}[\cite{bew}] \label{thm:cont_lyap}
If $\sum_{i=1}^k \lambda_i \geq M$ for a dense set of Teichm\"uller curves
in some connected component stratum $\cHH^{*}(\kappa)$ of the moduli space of 
Abelian differentials $\cHH(\kappa)$, 
then the Lyapunov exponents $\lambda_i(\kappa)$ for the Teichm\"uller
geodesic flow on the entire component $\cHH^{*}(\kappa)$ also 
satisfy $\sum_{i=1}^k \lambda_i(\kappa) \geq M$.
\end{Thm}
\noindent
This theorem applies also to any $\GL_2^+(\RR)$-invariant suborbifold 
that contains a dense set of Teichm\"uller curves.
\par
\smallskip
In \cite{kontsevichzorich} the two authors conjectured the large genus limit of the
Lyapunov is
$$ \lim_{g \to \infty} \lambda_2 \=1 $$
for the hyperelliptic components of the strata $\cHH(2g-2)$ and $\cHH(g-1,g-1)$
and that for all other strata and their components
$$ \lim_{g \to \infty} \lambda_2 \= \frac12\,. $$
\par
This first part of this conjecture now follows. The proof of this 
corollary was given by F.~Yu, assuming the validity of
Theorem~\ref{thm:main_estimate} and Theorem~\ref{thm:cont_lyap}.
\par
\begin{Cor}[{\cite[Conjecture~5.13]{Fei}}] 
\label{cor:hyplim}
For the hyperelliptic components of the series of strata $\cHH(2g-2)$ 
and $\cHH(g-1,g-1)$ the large genus limits of Lyapunov exponents are
$$ \lim_{g \to \infty} \lambda_k \=1 $$
for any fixed $k \geq 1$.
\end{Cor}
\par
(The Lyapunov exponents in the preceding statement are defined for $g \geq k$.)
\par
\begin{proof}
For those hyperelliptic strata the Harder--Narasimhan filtration over any
Teichm\"uller curve is computed in \cite{yuzuo} to be given by the subbundles
$\cEE_k = f_* \omega_{\ol{X}/\ol{C}}(-(2g-2k)S)$ resp.\ 
$\cEE_k  = f_* \omega_{\ol{X}/\ol{C}}(-(g-k)(S_1+S_2))$
for $k=1,\ldots g$ and the degrees of the successive quotient line bundles 
$\cEE_k/\cEE_{k-1}$ are equal to 
\bes 
\deg(\cEE_k/\cEE_{k-1}) \= \frac{|\chi|}2 \biggl(1 - \frac{2(k-1)}{2g-1}\biggr) \quad \text{resp.}\quad  
\deg(\cEE_k/\cEE_{k-1}) \= \frac{|\chi|}2 \biggl(1 - \frac{(k-1)}{g}\biggr)\,.
\ees
The implies that $2\deg(\cEE_k)/|\chi|$ tends to~$k$ in both cases as $g$ tends
to infinity. Together with the Theorem~\ref{thm:cont_lyap} our main theorem 
implies the result.  
\end{proof}
\par
A similar statement holds for any family of hyperelliptic loci in a sequence of strata
where the order of at least one singularity tends to infinity. 
 
\section{Application: Lyapunov exponents for some 
hypergeometric groups and Calabi--Yau threefolds}  \label{sec:APPCY}

In this section we apply our main theorem to a class of VHS of rank greater than
one. Our example is the well-studied class of hypergeometric local systems
that arise from Calabi--Yau threefolds with $h^{2,1}=1$. The irreducible local systems
that meet the additional requirements imposed by physics (existence of a  
MUM-point and a conifold point, see Section~\ref{sec:CYandGen} for details)
depend on two parameters~$\mu_1, \mu_2$ called local exponents 
(see Section~\ref{sec:localexp} for the definition). For
any pair $0 < \mu_1 \leq \mu_2 \leq 1/2$ with $\mu_i \in \RR$  the corresponding
local system admits an $\RR$-VHS. We compute the degrees of the
Hodge bundles and, consequently, a lower bound for the Lyapunov exponents.
\par
\begin{Thm} \label{thm:degs}
Suppose that the local exponents $0 < \mu_1 \leq \mu_2 \leq 1/2$ at the point $z=\infty$
of a Calabi--Yau-type hypergeometric group with $h^{2,1}=1$
are $(\mu_1,\mu_2,1-\mu_2,1-\mu_1)$. Then the
degrees of the Hodge bundles are
\bes
\deg_{\rm par} \cEE^{3,0} \=  \mu_1 \quad \text{and} \quad \deg_{\rm par} \cEE^{2,1} \=
 \mu_2\,.
\ees
\end{Thm}
\par
For families of Calabi--Yau threefolds the local mondromies are quasi-unipotent, hence
$\mu_i \in \QQ$.
In Table~\ref{cap:mirror}  we reproduce from~\cite{DoranMorgan} the well-known list 
of possible parameters $(\mu_1,\mu_2)$ that meet the physically relevant
conditions together with approximations for the Lyapunov exponents. Explanations
for the first three columns are given in Section~\ref{sec:CYandGen}. 
\par
The most remarkable conclusion from the numerical approximation of Lyapunov
exponents is the following. In the first seven cases the sum of Lyapunov
exponents matches the lower bound predicted by Theorem~\ref{thm:main_estimate}. 
The table lists the corresponding sum as exact fractions, but note that
only three digits seem to be reliable in the experiments. In the remaining cases, 
the sum $\lambda_1 + \lambda_2$ 
of Lyapunov exponents is strictly larger than predicted by the lower
bound in Theorem~\ref{thm:main_estimate}. Note that in precisely the $7$~cases
of (numerical) equality the monodromy groups of the hypergeometric local
systems are of infinite index (``thin'') in $\Sp(4,\ZZ)$ while in the
other  $7$~cases the monodromy group is of finite index in $\Sp(4,\ZZ)$.
This follows from combining the results in~\cite{BravThomas} and \cite{SinghVenk}.
It would be interesting to decide if in these seven cases actually 
equality holds and to explain the relation to the arithmeticity of
the monodromy groups. We provide  further conjectures in this direction in
Section~\ref{sec:conj} below.
\par
\begin{table}[h]
$$\begin{array}{|l|l|l||l|l|l|l|l|}
\hline
\# & \text{Model} & C & d & \mu_1, \mu_2 & \lambda_1 & \lambda_1 + \lambda_2 & -\chi  \\
\hline
1 &  & 46 & 1 & 1/12, 5/12 & 0.97 & 1 & 11/12  \\
\hline
2 &  & 44 & 2 & 1/8, 3/8 & 0.95 & 1 & 7/8 \\
\hline
3 &  & 52 & 4 & 1/6, 1/2 & 1.27 & 4/3 & 1 \\
\hline
4 & \PP^4[5] & 50 & 5 & 1/5, 2/5 & 1.12 & 6/5 & 4/5  \\
\hline
5 &  & 56 & 8 & 1/4, 1/2 & 1.40 & 3/2 & 1  \\
\hline
6 & \PP^6[2^2,3] & 60 & 12 & 1/3, 1/2 & 1.53 & 5/3 & 1 \\
\hline
7 & \PP^7[2^4] & 64 & 16 & 1/2, 1/2 & 1.75 & 2 & 1  \\
\hline
\hline
8 &  & 22 & 1 & 1/6, 1/6 & 0.75 & 0.92 & 1  \\
\hline
9 &  & 34 & 1 & 1/10, 3/10 & 0.77 & 0.83 & 9/10 \\
\hline
10 &  & 32 & 2 & 1/6, 1/4 & 0.84 & 0.97 & 11/12 \\
\hline
11 &  & 42 & 3 & 1/6, 1/3 & 0.96 & 1.06 & 5/6 \\
\hline
12 &  & 40 & 4 & 1/4, 1/4 & 1.07 & 1.30 & 1 \\
\hline
13 &  & 48 & 6 & 1/4, 1/3 & 1.15 & 1.31 & 11/12\\
\hline
14 &  & 54 & 9 & 1/3, 1/3 & 1.34 & 1.60 & 1 \\
\hline
\end{array} $$
\caption{Table of CY-VHS and approximate values of their Lyapunov exponents} 
\label{cap:mirror}
\end{table}
\par
Note that there is another commonly used normalization of the degrees 
and Lyapunov
exponents. Instead of working with parabolic degrees and over $\PP^1$ with
three singular points, we can view the above local systems as representations
of the Fuchsian triangle group~$\Delta(n,\infty,\infty)$, where~$n$ is the
least common multiple of the denominators of $\mu_1$ and $\mu_2$ if
$0<\mu_1<\mu_2<1/2$ and where $n=\infty$ if at least one of the inequalities
becomes an equality. Geometrically, 
this corresponds to viewing the local systems over the orbifold 
$C = \HH/\Delta(n,\infty,\infty)$. The orbifold Euler characteristic 
$-\chi$ of~$C$ is given in each case in the last column. Note that
$0<-\chi\leq 1$ in all the cases. One can also define and compute
Lyapunov exponents $\lambda_1^{\rm orb}, \lambda_2^{\rm orb}$ of the 
corresponding local systems over the orbifold 
$C = \HH/\Delta(n,\infty,\infty)$. They are related to the
Lyapunov exponents $\lambda_i$ over the thrice punctured sphere by
$$ \lambda_i \= \lambda_i^{\rm orb} \,\cdot \, |\chi|.$$
The corresponding orbifold degrees of Hodge bundles can be computed
as ordinary degrees of line bundles on a cyclic cover where all
the monodromies are unipotent, as indicated in Section~\ref{sec:CYandGen}.
The orbilfold normalization $\lambda_1^{\rm orb}, \lambda_2^{\rm orb}$ was
used in previous computations for Teichm\"uller curves (e.g.\ in 
 \cite{bouwmoel} and \cite{cyclicekz}).
\par
\subsection{Hypergeometric differential equations} \label{sec:HGDE}

We fix two sequences of real numbers $\bfal = (\alpha_1,\ldots, \alpha_n)$
and $\bfbe = (\beta_1,\ldots, \beta_n)$ with
\ba
0 &\leq \alpha_1 \leq \cdots \leq \alpha_n < 1 \\
0 &\leq \beta_1 \leq \cdots \leq \beta_n < 1\\
\ea
and with the property $\alpha_i \neq 1-\beta_j$. The regular hypergeometric 
differential operator is the operator
\be
 P \= P(\bfal,\bfbe) \= \prod_{i=1}^n (D-\alpha_i) - t 
\prod_{i=1}^n (D-\beta_i),
\qquad D = t\,\frac{d}{d t}  
\ee
It gives rise to a flat connection $\nabla$ on the trivial vector bundle $\cVV_0$ on $\PP^1$
with regular singularities precisely at the points $\{0,1,\infty\}$. We 
refer to this local system as the hypergeometric local system $\VV = \VV(\bfal,\bfbe)$.
\par
A {\em hypergeometric group} with parameters $\bfa = (a_1,\ldots,a_n)$ and 
$\bfb = (b_1,\ldots,b_n)$ 
subject to the conditions $|a_i| = 1 = |b_j|$ and $a_i \neq 1/b_j$ for all $(i,j)$ 
is a subgroup of $\GL_n(\CC)$ generated by three elements
\be h_0,\,h_1,\, h_\infty \in \GL_n(\CC) \quad \text{with} \quad h_\infty h_1 h_0 \= {\rm Id} \ee
such that 
\be \det(X\,{\rm Id} -h_\infty) \= \prod_{i=1}^n (X-a_i), \quad 
\det(X\,{\rm Id} -h_0^{-1})  \= \prod_{i=1}^n (X-b_i)
\ee
and such that $h_1$ is a pseudo-reflection. Here, a {\em pseudo-reflection} 
is an element $g \in \GL_n(\CC)$ such that $g-{\rm Id}$ has rank one.
\par
Up to conjugation there is a unique hypergeometric group for a given set 
of parameters. The proof due to Levelt and monodromy matrices can be
found e.g.\ in \cite{BeHe}, Theorem~3.5.
The hypothesis $a_i \neq 1/b_j$ guarantees that the flat bundle 
$\VV$ is irreducible (\cite{BeHe}, Proposition~3.3). 
\par
The monodromy group of $\VV$ is the hypergeometric group with parameters $\bfa$ 
and~$\bfb$ where $e^{2\pi i \alpha_j} = a_j$
and $e^{2\pi i \beta_j} = b_j$ for $j=1,\ldots,n$.
\par

\subsection{Simpson's correspondence in the parabolic case} \label{sec:Simp}

In order to state Simpson's correspondence we need to extend the definition
of parabolic structure and stability from vector bundles to the cases of
parabolic vector bundles, local systems and Higgs bundles, respectively.
\par
A {\em regular parabolic Higgs bundle} is a parabolic vector bundle $(\cEE,F^\bullet)$
together with a Higgs field, i.e.\ a map of sheaves of $\cOO_C$-modules
\be \theta: \cEE  \to \cEE \otimes \Omega^1_{C} \ee
that respects the parabolic structure in the sense that for every $c \in \Delta$ 
the map $\theta$ extends for every $\alpha \in [0,1)$ to
\be \label{eq:compHiggsFilt}
\theta_{c,\alpha}: \cEE_c^{\geq \alpha} \to 
\cEE_c^{\geq \alpha} \otimes \Omega^1_{\ol{C}}(\Delta)\,.
\ee
\par
A {\em regular parabolic system of Hodge bundles} is a regular parabolic Higgs 
bundle whose underlying vector bundle admits a decomposition 
$\cEE = \oplus_{p \in \ZZ} \cEE^p$, such that $\theta$ has degree $-1$ with respect 
to the grading given by this decomposition.
\par
\smallskip
Recall that a vector bundle $\cVV$ is called stable, if for every 
subbundle $\cMM \subset \cVV$ the condition
\be \label{eq:defstab}
\frac{\deg(\cMM)}{\rank(\cMM)} \,< \,\frac{\deg(\cVV)}{\rank(\cVV)}
\ee
holds. Similarly, a {\em parabolic vector bundle} (resp.\ {\em a local system}, resp.\ 
a {\em Higgs bundle}) is called {\em stable}, if the condition~\eqref{eq:defstab} holds
for every parabolic subbundle (resp.\ every subbundle preserved by the connection, 
resp.\ every subbundle preserved by the Higgs field).
\par
\medskip
Simpson's correspondence (\cite{SiNoncomp}) for the non-compact case states that there 
is a natural one-to-one correspondence between stable regular parabolic Higgs bundles 
and  stable parabolic local systems of degree zero.
\par
There is an action of $\CC^*$ on the set of regular parabolic Higgs bundles of 
degree zero, where $s \in \CC^*$ sends $(E,\theta)$ to $(E,s\theta)$ while 
preserving the filtration. Fixed points of this action are precisely the regular 
parabolic systems of Hodge bundles of degree zero.
\par
\medskip
Since hypergeometric local systems are rigid (e.g.\ \cite{BeHe},
Proposition~3.5), Simpson's correspondence implies the following
(\cite{SiNoncomp}, Corollary~8.1).
\par
\begin{Cor}
A hypergeometric local system $\VV = \VV(\bfal,\bfbe)$ 
carries a complex variation of Hodge structures.
\end{Cor}
\par
The Hodge numbers, i.e.\ the ranks $h^p$ of the summands $\cEE^{p}$, are known
by a theorem of Fedorov. If we set $\rho(k) = \#\{j: \alpha_j < \beta_k\} -k$, 
then the main theorem of \cite{Fed} (Theorem~1) states that
\be 
h^p \= \# \rho^{-1}(p)
\ee
after an appropriate shifting of the weight (or the grading).
\par

\subsection{Families of Calabi--Yau threefolds with $h^{2,1}=1$ and 
generalizations.} \label{sec:CYandGen}

Families of Calabi--Yau threefolds with $h^{2,1}=1$ carry a weight~$3$
variation of Hodge structures and by definition of Calabi--Yau threefolds
the Hodge numbers of these families are $(1,1,1,1)$, i.e.\ $\dim \cEE^{p,q} = 1$
for $p=0,1,2,3$. In a VHS arising from geometry the VHS has an $\RR$-structure
and quasi-unipotent monodromies. Motivated by physics requirements
the most intensely investigated families satisfy the following properties.
They are over $\PP^1$, smooth outside three points, have one point of 
{\em maximal unipotent monodromy} (MUM, i.e.\ there
is only one Jordan block of maximal size) and one rank one unipotent point.
There are 14 possible cases, as derived in \cite{DoranMorgan}. 
They are given in Table~\ref{cap:mirror}. In some cases, these families have
been realized geometrically (e.g.\ as complete intersection in weighted
projective spaces) and the first column of the table lists this
realization (if available, e.g.\ $\PP^4[5]$ refers to the (mirror) quintic)). 
\par
The local exponents of such a hypergeometric system, with real structure, 
with a MUM-point and with a point where the monodromy is unipotent of rank one, are
\ba
\bfbe = (0,0,0,0)  & &\quad \text{at} \quad& t=0 \\
(0,1,1,2) &&\quad \text{at} \quad& t=1 \\
\bfal = (\mu_1, \mu_2,1-\mu_2, 1-\mu_1) && \quad\text{at} \quad& t=\infty\,, \\
\ea
see e.g.\ \cite{YoFuchs}, \cite{YoHyp},
\cite{BeHe}, \cite{Fed} for general background.
\par
A realization of monodromy groups of the hypergeometric local
systems listed in Table~\ref{cap:mirror} is given by
$$ T_0 \= \left(\begin{smallmatrix}  
1 & 0 & 0 &0 \\
1 &  1 & 0  & 0 \\
1/2 & 1 & 1 & 0 \\
1/6 & 1/2 & 1 & 1 \\
\end{smallmatrix}\right), \quad 
T_1 \= \left(\begin{smallmatrix}  
1 & -C/12 & 0 &-d \\
0 &  1 & 0  & 0 \\
0 & 0 & 1 & 0 \\
0 & 0 & 0 & 1 \\
\end{smallmatrix}\right) $$
with the parameters $(C,d)$ as in the table. Here, the symplectic form 
defining the polarization of the Hodge structure on~$\VV(\bfal,\bfbe)$
is given by $$ \Omega \= \left(\begin{smallmatrix}  
0 & C/12 & 0 &d \\
-C/12 &  0 & -d  & 0 \\
0 & d & 0 & 0 \\
-d & 0 & 0 & 0 \\
\end{smallmatrix}\right)$$
and this symplectic form can be conjugated into $\Sp(4,\ZZ)$. 
The proof of Theorem~\ref{thm:degs} does not use properties of these 
realizations. In fact, the representation is real by \cite{Fed}, Theorem~2 if  
\bes \alpha_m + \alpha_{4+1-m} \in \ZZ \quad \text{and} \quad 
\beta_m + \beta_{4+1-m} \in \ZZ\,. \ees
\par
\medskip
The basic principle for the proof of Theorem~\ref{thm:degs} is the following. We
consider the Kodaira--Spencer maps  (graded pieces of the Higgs fields)
\be \label{eq:taup}
\tau_{p-1}: 
\cEE^{p,q} \to \cEE^{p-1,q+1} \otimes \Omega^1_{\ol{C}}(\Delta)\,. \ee
In our situation, these are maps between line bundles. The maps
$\tau_0$, $\tau_1$ and $\tau_2$ are non-zero by
Lemma~\ref{HiggsReg} below, hence inclusions. To compute the (parabolic) 
degrees it suffices to compute the length of the cokernels of these maps
and to determine the parabolic structure. We prove the following lemma
(which applies not only to hypergeometric systems, but to any self-dual
flat bundle) and explain the notions about differential equations 
in the next subsection.
\par
\begin{Lemma} \label{HiggsReg}
If $x \in C$ is a regular point of the local system $\VV$ on $\ol{C}$, 
then all the Kodaira--Spencer maps $\tau_i$ are isomorphisms at $x$.
\par
More generally, if $x \in C$ and the local exponents
$\mu_1 < \mu_2 < \mu_3 < \mu_4$ are distinct and integral, 
then $\tau_0$ has a cokernel of length $\mu_2  - \mu_1 -1$, 
and so does $\tau_2$ by duality.
The map $\tau_1$ has a cokernel of length $\mu_3 - \mu_2 -1$.
\par
If $c \in \Delta \subset \ol{C}$ and with local exponents
$\mu_1 \leq \mu_2 \leq \mu_3  \leq \mu_4$, 
then $\tau_0$ has a cokernel of length 
$\lfloor\mu_2\rfloor-\lfloor\mu_1\rfloor$, and so does $\tau_2$ by duality.
The map $\tau_1$ has a cokernel of length $\lfloor\mu_3\rfloor-\lfloor\mu_2\rfloor$.
\end{Lemma}
\par
\begin{proof}[Proof of Theorem~\ref{thm:degs}] 
We can apply the first observation in Lemma~\ref{HiggsReg} to every point 
different from $0,1,\infty$ and we can apply  the observation for boundary 
points in this Lemma at the MUM-point $t=0$ and to the
point $t = \infty$ to conclude that at all these points all the $\tau_i$
are isomorphisms. Finally, the last statement in Lemma~\ref{HiggsReg} tells 
us that at the unipotent rank one point at $t=1$, the maps $\tau_1$ is still an
isomorphism, while $\tau_0$ and $\tau_2$ have cokernels of length one.
\par
Next, we consider the $[0,1)$-filtrations, which are non-trivial only at
the point $t=\infty$. There, since all the $\cEE^{p,q}$ are line bundles, 
the only possibility of a filtration respecting the regularity
hypothesis~\eqref{eq:compHiggsFilt} and the fact that for a system of
Hodge bundles~$\theta$ shifts the degree by~$-1$ is
$$ V_\infty^{\geq \mu_i} \= \oplus_{p=4-i}^3 \cEE_\infty^{p,3-p}\,.$$
We deduce from  properties of an $\RR$-VHS that 
$\deg_{\rm par}\cEE^{p,q} = - \deg_{\rm par} \cEE^{q,p}$. Hence
the fiber at $t=\infty$ of $\cEE^{p,3-p}$ is the graded piece of 
weight $\mu_{4-p}$ of the filtration. This implies that 
$$  - \deg(\cEE^{2,1}) \= \deg(\cEE^{1,2}) + 1, \quad - \deg(\cEE^{3,0}) \= 
\deg(\cEE^{0,3}) + 1\,. $$
Since $\tau_1$ is an isomorphism, we conclude from~\eqref{eq:taup} that 
$\deg(\cEE^{2,1}) \= \deg(\cEE^{1,2}) + 1$ and hence $\deg(\cEE^{2,1}) = 0$
Since $\tau_0$ has a cokernel of length~$1$ we conclude from~\eqref{eq:taup} 
again $\deg(\cEE^{3,0}) = 0$. This gives the parabolic degrees as claimed 
in the theorem.
\end{proof}

\subsection{Local exponents, weight filtration and the cokernel lemmas} 
\label{sec:localexp}

We need two general concepts about a flat bundles $\VV$, local
exponents and the monodromy weight filtration. We let $n = {\rm rank}(\VV)$
and later we specialize to the case $n=4$ of primary interest.
\par
To recall the definition and the properties of local exponents, fix a
point $c \in \ol{C}$, let $t$ be a coordinate of $C$
such that $c$ is the point $t=0$ and fix a section $\omega(t)$
of $\VV$, whose first $n$ derivatives (with respect to $\nabla_{d/dt}$)
generate $\VV$ in a neighborhood of~$c$. Then there are meromorphic 
functions $P_i(t)$ such that
$$ L(\omega) = (\nabla_{d/dt}^n + \sum_{i=0}^{n-1} P_i(t)\, \nabla_{d/dt}^i )(\omega) = 0$$ 
Since the local system $\VV$ is supposed to have regular singularities, 
$t^{n-1-i} P_i(t)$ is holomorphic at $0$. It will be convenient to 
rewrite the differential equation in terms of the differential operator~$D$
as
$$ \nabla_{d/dt}^n + \sum_{i=0}^{n-1} P_i(t)\, \nabla_{d/dt}^i 
\= \nabla_{D}^n + \sum_{i=0}^{n-1} Q_i(t)\, \nabla_{D}^i.$$
\par
Now consider in general a linear differential operator
\be \label{eq:DE}
L(y) \= \frac{d^n y}{d\, t^n}  + \sum_{i=0}^{n-1} 
Q_i(t)  \frac{d^i y}{d\, t^i} 
\ee
\par
The {\em local exponents} $\{\mu_1(c), \ldots, \mu_n(c) \}$ of $L$ 
at $c\in C$ are the solutions
of the equation
$$ y^n  + \sum_{i=0}^{n-1} Q_i(0) y^i \= 0\,.$$
The local exponents at a point~$c$ are well-defined up to a simultaneous
shift by some integer. This ambiguity is due to the possibility of
replacing the section $\omega(t)$ by $t^k\omega(t)$, see Frobenius'
theorem (e.g.\ in \cite{YoFuchs}) and below.
\par
A point $c \in C$ is called {\em regular} if the functions $P_i(t)$ 
are regular at $c$. The regular points are precisely those points
where the local exponents are of the form $\{k, k+1, \ldots, k+n-1\}$
for some $k$.
\par
\medskip
The local exponents $\{\mu_1(c), \ldots, \mu_n(c) \}$ 
determine the exponents needed to write local solutions
of the differential equation as a power of uniformizer times a power series 
expansion. More precisely, if the difference of any two local exponents 
is non-integral, then the theorem of Frobenius states that the solutions of the
differential equation~\eqref{eq:DE} are
\bes 
s_i \= t^{\mu_i} P_i \quad \text{with} \quad P_i \in 1 + \C[[t]]\,.
\ees
We refer to this basis of solutions as {\em Frobenius basis}. If some
difference of local exponents is integral, then one has to add logarithmic
terms, according to the monodromy at~$c$. We give an example for $n=4$. 
\par
Suppose that the monodromy is maximal unipotent (hence all the $\mu_i$ 
are the same). Then the solutions are of the form
\bas s_1 &\= t^{\mu_1} P_1\,, \\
s_2 &\= \log(t) s_1 \+ t^{\mu_2} P_2 \\
s_3 &\= \tfrac12 \log(t)^2 s_2 \+ \log(t) s_1 \+ t^{\mu_3} P_3 \\
s_4 &\= \tfrac16 \log(t)^3 s_3 \+ \tfrac12 \log(t)^2 s_2 
\+ \log(t) s_1 \+ t^{\mu_4} P_4 \,.\\
\eas
We deduce that, by definition, a basis of local sections of the  the Deligne extension 
is  $t^{\mu_1}P_1, t^{\mu_2}P_2, t^{\mu_3}P_3, t^{\mu_4}P_4$. In fact, 
this last conclusions holds for any local monodromy matrix. 
For this reason the proof of Lemma~\ref{HiggsReg} 
does not depend on the form of the monodromy matrix.
\par
\medskip
We have expressed above the local exponents in terms of a (polynomial associated
to a) differential operator $L$, which in turns depends on the choice of
a local section $\omega$ of $\VV$. We recall how to retrieve $(\VV,\omega)$
up to isomorphism from $L$. In fact, let $\Sol \subset \cOO_C$ be 
the rank-$n$ local system of solutions of $L$. Then $\Sol \cong \VV^\vee$, 
since in fact the multiplication map
$$ m:\Sol  \otimes \cOO_C \to \cOO_C,$$
defines a section of $\VV$ and the pair $(\Sol,m)$ is isomorphic to the pair
$(\VV,\omega)$ we started with.
\par
In terms of a basis of $\Sol$ and its dual basis we can  compute the effect of
the covariant derivative. To simplify notations, we restrict to the case $n=4$
of primary interest here. Let $\{s_1,s_2,s_3,s_4\}$ be a basis of $\Sol$ around~$c$ 
and denote by
$$  s_j^\vee:\sum_{i=1}^4 s_i \otimes g_i \mapsto g_j 
\in \Sol^\vee \cong \VV \qquad (g_i \in \cOO_C(U) \,\,\text{for some $U$})$$
the elements of the dual basis. In this basis $m = \sum_{i=1}^4 s_i s_i^\vee$, 
as a section of $\Sol^\vee \cong \VV$. Moreover, 
\bas
\nabla_{d/dt}(m)\, \Bigl(\sum_{i=1}^4 s_i \otimes g_i \Bigr) &\= 
d\,\Bigl(\sum_{i=1}^4 s_i \otimes g_i \Bigr) - \sum_{i=1}^4 s_i g_i
\= \sum_{i=1}^4 s'_i \otimes g_i, \\ \quad \text{i.~e.}
\quad  \nabla_{d/dt}(m) &\= \sum_{i=1}^4 s'_i s_i^\vee\,.
\eas
\par
This completes the preparation for the main lemma.
\par
\begin{proof}[Proof of Lemma~\ref{HiggsReg}]
We start with the case of a regular point. Without changing the
length of the cokernels we may choose the section~$\omega$ to be
non-vanishing at~$c$, hence the local exponents are $\{\mu_1=0,1,2,3\}$. 
The length of the cokernel
of $\tau_0$ at the point $c$, i.~e.\ at $t=0$, is the vanishing order of 
\be \label{eq:derivm}
\nabla_{d/dt}(m) \= \sum_{i=1}^4 s_i' s_i^\vee \,\,\in \,\,V / \langle m \rangle,
\ee
where $V = \langle s_1^\vee,s_2^\vee,s_3^\vee,s_4^\vee \rangle$ is the fiber of 
$\cVV$ over~$c$. We use the Frobenius basis $\{s_1,s_2,s_3,s_4\}$ from now on. 
Consider the matrix $M = M(t)$ with entries $M_{ij}(t) = 
s_i^{(j-1)}(t)$. Since the $2\times2$-minor $M_{12}^{12}$ of $M$ has a determinant
with non-zero constant term (considered as element of $\CC[t]$), the vanishing 
order of~\eqref{eq:derivm} is zero, i.e.\ the map~$\tau_0$ is an isomorphism at~$c$. 
Similarly, the minor $M_{123}^{123}$ and also the determinant~$M$ itself 
have non-zero constant terms by our hypothesis on the local exponents.
Since 
$$\nabla^{(j)}_{d/dt}(m) \= \sum_{i=1}^4 s_i^{(j)} s_i^\vee, $$
this is precisely what we need to deduce that also the Kodaira--Spencer 
maps $\tau_1$ and $\tau_2$ are isomorphisms at~$x$. 
\par
The case of general (but still integral) local exponents, follows similarly. 
In fact, the minimal order of vanishing of a $2\times2$-minor of the first two 
rows of $M$ is given by
$M_{12}^{12}$, which starts with $t^{\mu_2-1}$.  Hence the length of the cokernel
of $\tau_0$ is as claimed. The minor $M_{123}^{123}$ starts with  $t^{\mu_3-2}$.
This is the length of the cokernel of the composition of Kodaira--Spencer maps
$\tau_1 \circ \tau_0: \cEE^{3,0} \to \cEE^{1,2} \to \Omega^1_{\ol{C}}(\Delta)^{\otimes 2}$ and it implies the claim about $\tau_1$. The same argument with the
determinant~$M$ finally implies the claim about~$\tau_2$.
\par
The discussion so far was concerned with points  $c \in C$. If 
$c \in \Delta \subset \ol{C}$, then the calculations above are the same with~$M$ 
replaced by the matrix with entries $M_{ij}(t) = \left(t\tfrac{\partial}{\partial t} 
\right)^{j-1} s_i(t)$.
This increases the length of each of the cokernels by one with respect to the 
previous calculations. 
\par
Finally in the case of non-integral local exponents recall that the 
sections of the Deligne extension are given by $t^{\{\mu_i\}} s_i$ in
terms of the Frobenius basis, where $\{\mu\} = \mu - \lfloor \mu \rfloor$ 
denotes the fractional part of $\mu$. Consequently, the preceding calculation
applies again, now with $\mu_i$ replaced by $\lfloor \mu_i \rfloor$.
\end{proof}
\par

\subsection{Conjectural region of equality} 
\label{sec:conj}

It seems likely, that the seven cases of Calabi--Yau type families with
equality are not isolated examples. Initially we conjectured that
the  equality $\lambda_1 + \lambda_2 \= 2(\mu_1 + \mu_2)$ is attained
in the entire region in the $(\mu_1,\mu_2)$-plane defined by the linear
inequality $3 \mu_2 \geq \mu_1 + 1$. After more detailed numerical
experiments by Fougeron  the conjecture cannot be
upheld in this form any more (see \cite{Fougeron}). According to 
these experiments it rather appears that equality is attained at
an infinite number of rational points in the $(\mu_1,\mu_2)$-plane. 
It would be very interesting to relate in general thinness of the 
monodromy group and the equality $\lambda_1 + \lambda_2 \= 2(\mu_1 + \mu_2)$,
see also \cite{dander} for some results in this direction.
\par
%
\par
Special cases of the conjecture can be equivalently formulated 
as a number-theoretic problem.  This new hypothetical  non-vanishing
property is  similar to the   non-vanishing
of the classical modular form
$\Delta(q):=q\prod_{n\ge 1} (1-q^n)^{24}$ for $0<|q|<1$.
Consider the mirror quintic (case~$4$
in Table~\ref{cap:mirror}), normalized so that the MUM-point is zero, 
the conifold point is at $t=\infty$ and the remaining singular point is $t=(1/5)^5$
instead of $t=1$ before. Since strict inequality
in Theorem~\ref{thm:main_estimate} is caused by the presence of bad points (see 
\eqref{eq:Tbad}), it 
seems natural to look for a flat section of $p^*(\wedge^2 \VV)$ (where
$p: \HH \to C$ is the universal cover) that avoids the 
bad locus. As a first attempt we take~$L$ to be the Lagrangian $2$-plane that is 
invariant under the monodromy around $t=0$ and the flat section it defines by
parallel transport along the upper half plane. In fact, since the condition
of having empty bad locus is open, it suffices to find a single flat section
of $p^*(\wedge^2 \VV^\vee)$ such that the pairing with~$L$ is everywhere non-zero
on~$\HH$. Here, again, we try the $2$-plane invariant under the monodromy 
around $t=0$ and its parallel transport. Near $t=0$, the $2$-plane~$L$
is generated by the differential $3$-forms $\{\Omega^{3,0}$ and $\Omega^{2,1}\}$ 
generating $\cEE^{3,0}$ and $\cEE^{2,1}$ respectively. The homology $2$-plane
is generated by the two ``shortest'' $3$-cycles $\{\gamma_0,\gamma_1\}$. 
It is well-known (e.g.\ \cite{kont_ICM}) that 
$$ \psi_0(t) \,:=\, \int_{\gamma_0} \Omega^{3,0}(t) \= \sum_{n \geq 0} \frac{(5n)!}{n!^5}t^n $$
and 
$$ \psi_1(t) \,:=\, \int_{\gamma_1} \Omega^{3,0}(t) \=  \log(t) \psi_0 \,+\, 
\sum_{n \geq 0} \frac{(5n)!}{n!^5}  \Biggl(\sum_{k=n+1}^{5n} \frac1k \Biggr) t^n
\,.$$
Since the Kodaira--Spencer map is non-vanishing (on $\PP^1 \setminus \{0,(1/5)^5,
\infty\}$), the integral against $\Omega^{2,1}(t)$ is given by the $t$-derivatives
of $\psi_0$ and $\psi_1$ respectively. Consequently, the contraction of~$L$ against
$\langle \gamma_0,\gamma_1 \rangle$ is given by the Wronskian
$$ W(t) \= \psi_0(t) \psi_1'(t) - \psi_0'(t) \psi_1(t)\,.$$ 
We consider the composition $F(q) = W \circ \lambda(q)$ with the $\lambda$-function
$$ \lambda: \Delta^* \to \CC, \quad \lambda(q) = \frac{q}{5^5} \cdot \Biggl( \frac{\sum_{n \in \ZZ} q^{n^2+n}}
{\sum_{n \in \ZZ} q^{n^2}} \Biggr)^4 \,,$$
where $\Delta^* = \{q \in \CC : 0<|q|<1\}$ denotes the punctured unit disc.
By the choice of~$L$ and $\{\gamma_0,\gamma_1\}$, the function~$F$ extends
meromorphically with a simple pole across $q=0$.
\par
Altogether, the non-vanishing of $L$ contracted against 
$\langle \gamma_0,\gamma_1 \rangle$ on the whole upper half plane and 
the mirror quintic case $\mu_1=1/5$, $\mu_2=2/5$ of the conjecture
stated in the introduction follows from the following statement.
\par
\begin{Conj} \label{conj:mirrorquintic}
The pullback~$F$ of the Wronskian $W(t)$ via $\lambda$ vanishes nowhere 
on the puntured unit disc~$\Delta^*$.
\end{Conj}
\par
Strong numerical evidence for this conjecture is given by considering
the growth rate of the coefficients of $1/F$. They appear to be growing
like $\exp(C\sqrt{n})$ for some~$C$, whereas for the reciprocal of  function 
with a zero in the disc (like e.g.\ $1/\psi_0(\lambda(q))\,$) has radius
of convergence strictly smaller than one and coefficients growing like $\exp(n)$.
\par

\begin{appendix}

\section{The multiplicative ergodic theorem and  equivalent norms for
measurable cocycles}

Suppose that we have a smooth, or continuous (or just measurable)
finite-dimen\-sional complex vector bundle $\mathcal{V}$ 
of rank~$r$ over the base $B$,
where the smooth (or topological) manifold~$B$ is endowed with a
probability measure $\mu$. Suppose that a map $T:B\to B$ ergodic with
respect to the measure $\mu$ extends to a smooth (continuous,
measurable) automorphism $A$ of the vector bundle $\mathcal{V}$. In
other words, we suppose that the map~$T$ of the base to itself lifts
to a map $A$ of the total space of the vector bundle to itself
preserving the bundle structure, such that $A$ is fiberwise $\CC$-linear,
and such that the induced linear transformations $A_x:
\mathcal{V}_{(x)}\to \mathcal{V}_{T(x)}$ of the fibers is invertible
for any $x\in B$. Suppose finally that each fiber $\mathcal{V}_{(x)}$
of the vector bundle $\mathcal{V}$ is endowed with a norm $\|\
\|_{(x)}$ which depends smoothly (continuously, measurably) on the
base point $x\in B$.

Consider the usual operator norm
$$
\|A_x\|:=
\max_{\vec v\in\mathcal{V}_{x}\setminus \vec 0}
\frac{\|A_x\vec v\|_{(T(x))}}{\|\vec v\|_{(x)}}\,.
$$
Define $\log^+(y)=\max(0,\log(y))$.

\begin{Defi} \label{def:meascocycle}
The above data $(B,T,\mu,\mathcal{V},\|\ \|,A)$ defines a
\textit{measurable cocycle} if $\log^+\|A_x\|$ is integrable over $B$
with respect to the measure $\mu$,
$$
\int_B \log^+\!\|A_x\|\, d\mu(x) <\infty\,.
$$
\end{Defi}
\par
We state the Multiplicative Ergodic Theorem in a form close to
the original formulation in~\cite{Oseledets}.

\begin{Thm}[Oseledets Theorem] Suppose that $(B,T,\mu,\mathcal{V},\|\ \|,A)$
is an integrable cocycle. Then there exist real numbers $\lambda_{(1)} 
> \lambda_{(2)} > \cdots > \lambda_{(k)}$
and $T$-equivariant complex 
subbundles of $\mathcal{V}$ defined for almost every $x \in B$,
denoted by
$$ 0 \subsetneq \mathcal{V}^{\leq \lambda_{(k)}} \subsetneq \cdots  
\subsetneq \mathcal{V}^{\leq \lambda_{(1)}} \= \mathcal{V}\,,$$
such that for vectors $v \in \mathcal{V}^{\leq \lambda_{(i)}} \setminus \mathcal{V}^{\leq \lambda_{(i+1)}}$
we have
$$ \lim_{N \to \infty} \frac{1}{N} \log \| T^N(v) \| \, \to \, \lambda_{(i)} \,.$$
\end{Thm}
\par
We also use the notation $\lambda_1 \geq \lambda_2 \geq \cdots 
\geq \lambda_r$ for the Lyapunov spectrum consisting of the 
numbers $\lambda_{(i)}$ from Oseledets Theorem repeated with multiplicity 
${\rm rank}(\mathcal{V}^{\leq \lambda_{(i)}} / \mathcal{V}^{\leq \lambda_{(i+1)}})$.
\par
Instead of a discrete ergodic transformation of the vector bundle one
can consider an ergodic flow~$g_t$ on the base~$B$ and a smooth
(continuous, measurable) connection~$\nabla$ on the vector bundle,
where $\nabla$ is not assumed to be necessarily flat. Denote by
$A(x,t): \mathcal{V}_{(x)}\to \mathcal{V}_{g_t(x)}$ the linear
transformation of the fibers induced by the holonomy along the
trajectory of the flow.

\begin{Defi}
The cocycle $(B,g_t,\mu,\mathcal{V},\nabla,\|\ \|)$ is called\
{\em integrable}  if the function $\sup_{t\in[-1,1]}\log^+\|A(x,t)\|$ is
integrable over $B$ with respect to the measure $\mu$, i.e.\ 
$$
\int_B\ \ \sup_{t\in[-1,1]}\log^+\!\|A(x,t)\|\, d\mu(x) <\infty\,.
$$
In this situation we also say that $(\VV,g_t,\|\ \|)$ is an {\em integrable
flat bundle}.
\end{Defi}
\par
\par
The Multiplicative Ergodic Theorem stated above generalizes naturally
to multiplicative cocycles over flows.

It is clear from the definition that integrability of the cocycle and
the Lyapunov spectrum do not depend on the choice of the norm in the
vector bundle for a large class of norms. To provide a convenient
sufficient condition of equivalence of norms we start with the
following definition.

\begin{Defi}
Let $\mathcal{V}$ be a vector bundle over the base $B$; let $\mu$ be
a probability measure on $B$. We say that two norms $\|\ \|_1$ and
$\|\ \|_2$  on the vector bundle $\mathcal{V}$ are
$L^1(\mu)$-equivalent if the quantity
\begin{equation}
\label{eq:ratio:of:norms}
\max_{\vec v\in\mathcal{V}_{x}\setminus \vec 0}
\left|\log\frac{\|\vec v\|_{2\,(x)}}{\|\vec v\|_{1\,(x)}}\right|
=
\max_{\vec v\in\mathcal{V}_{x}\setminus \vec 0}
\left|\log\frac{\|\vec v\|_{1\,(x)}}{\|\vec v\|_{2\,(x)}}\right|
\end{equation}
is integrable over $B$ with respect to the measure $\mu$,
\begin{equation}
\label{eq:integral:of:ratio:of:norms}
\int_B \max_{\vec v\in\mathcal{V}_{x}\setminus \vec 0}
\left|\log\frac{\|\vec v\|_{2\,(x)}}{\|\vec v\|_{1\,(x)}}\right|\, d\mu(x) <\infty\,.
\end{equation}
\end{Defi}

The relation of $L^1(\mu)$-equivalence is, clearly, reflexive, symmetric, and
transitive.

\begin{Thm}
\label{th:equivalent:norms}
Suppose that data $(B,T,\mu,\mathcal{V},\|\ \|_1,A)$ define a
measurable cocycle. For any norm $\|\ \|_2$ which is $L^1(\mu)$-equivalent to
$\|\ \|_1$ the data $(B,T,\mu,\mathcal{V},\|\ \|_2,A)$ also define a
measurable cocycle, and it has the same Lyapunov filtration and the same
Lyapunov exponents as the original one.

Suppose that data $(B,g_t,\mu,\mathcal{V},\nabla,\|\ \|_1)$ define a
measurable cocycle. For any norm $\|\ \|_2$ which is $L^1(\mu)$-equivalent to
$\|\ \|_1$ the data $(B,g_t,\mu,\mathcal{V},\nabla,\|\ \|_2)$ also
define a measurable cocycle, and it has the same Lyapunov filtration and
the same Lyapunov exponents as the original one.
\end{Thm}
\begin{proof}
We prove the Theorem for the cocycle with the discrete time; the
proof for the cocycles with continuous time is completely analogous.

\begin{multline*}
\log^+ \max_{\vec v\in\mathcal{V}_{x}\setminus \vec 0}
\frac{\|A_x\vec v\|_{2\, (T(x))}}{\|\vec v\|_{2\, (x)}}
=\\=
\log^+ \max_{\vec v\in\mathcal{V}_{x}\setminus \vec 0}
\frac{\|A_x\vec v\|_{2\, (T(x))}}{\|A_x\vec v\|_{1\, (T(x))}}\cdot
\frac{\|A_x\vec v\|_{1\, (T(x))}}{\|\vec v\|_{1\, (x)}}\cdot
\frac{\|\vec v\|_{1\, (x)}}{\|\vec v\|_{2\, (x)}}
\\ \le
\log^+ \max_{\vec w\in\mathcal{V}_{T(x)}\setminus \vec 0}
\frac{\|\vec w\|_{2\, (T(x))}}{\|\vec w\|_{1\, (T(x))}}
+
\log^+\max_{\vec v\in\mathcal{V}_{x}\setminus \vec 0}
\frac{\|A_x\vec v\|_{1\, (T(x))}}{\|\vec v\|_{1\, (x)}}
+
\log^+ \max_{\vec v\in\mathcal{V}_{x}\setminus \vec 0}
\frac{\|\vec v\|_{1\, (x)}}{\|\vec v\|_{2\, (x)}}
\\ \le
\max_{\vec w\in\mathcal{V}_{T(x)}\setminus \vec 0}
\left|\log
\frac{\|\vec w\|_{2\, (T(x))}}{\|\vec w\|_{1\, (T(x))}}\right|
+
\log^+\max_{\vec v\in\mathcal{V}_{x}\setminus \vec 0}
\frac{\|A_x\vec v\|_{1\, (T(x))}}{\|\vec v\|_{1\, (x)}}
+
\max_{\vec v\in\mathcal{V}_{x}\setminus \vec 0}
\left|\log \frac{\|\vec v\|_{1\,(x)}}{\|\vec v\|_{2\, (x)}}\right|
\end{multline*}

It remains to note that since $T:B\to B$ is measure preserving we have
$$
\int_B \max_{\vec w\in\mathcal{V}_{x}\setminus \vec 0}
\left|\log\frac{\|\vec w\|_{2\,(T(x))}}{\|\vec v\|_{1\,(T(x))}}\right|\, d\mu(x)
=
\int_B \max_{\vec v\in\mathcal{V}_{x}\setminus \vec 0}
\left|\log\frac{\|\vec v\|_{2\,(x)}}{\|\vec v\|_{1\,(x)}}\right|\, d\mu(x)\,.
$$
Thus, the first and the third terms in the latter sum are
$L^1(\mu)$-integrable by definition of $L^1(\mu)$-equivalent norms
and the second term is $L^1(\mu)$-integrable since the cocycle
represented by the data $(B,T,\mu,\mathcal{V},\|\ \|_1,A)$ is
integrable by assumption of the Theorem. We have proved that
$L^1(\mu)$-equivalence of the norms $\|\ \|_1$ and $\|\ \|_2$ implies
that as soon as the cocycle represented by the data
$(B,T,\mu,\mathcal{V},\|\ \|_1,A)$ is integrable, the cocycle
represented by the data $(B,T,\mu,\mathcal{V},\|\ \|_2,A)$ is also
integrable. It remains to prove that the Lyapunov filtrations and the
Lyapunov spectra of the two cocycles coincide.

For almost all points $x\in B$ the Lyapunov filtrations and Lyapunov
exponents are well-defined for both cocycles and the ergodic sum
along the trajectory $x, T(x), T(T(x)), \dots$ of the
quantity~\eqref{eq:ratio:of:norms} converges to the
integral~\eqref{eq:integral:of:ratio:of:norms}. Namely, let
\begin{multline*}
a_N(x):=\frac{1}{N}
\left(
\max_{\vec v\in\mathcal{V}_{x}\setminus \vec 0}
\left|\log\frac{\|\vec v\|_{2\,(x)}}{\|\vec v\|_{1\,(x)}}\right|
+
\max_{\vec v\in\mathcal{V}_{T(x)}\setminus \vec 0}
\left|\log\frac{\|\vec v\|_{2\,(T(x))}}{\|\vec v\|_{1\,(T(x))}}\right|
+\dots \right.\\ \dots + \left.
\max_{\vec v\in\mathcal{V}_{T^{N-1}(x)}\setminus \vec 0}
\left|\log\frac{\|\vec v\|_{2\,(T^{N-1}(x))}}{\|\vec v\|_{1\,(T^{N-1}(x))}}\right|\right)
\end{multline*}
The Ergodic Theorem implies that  for almost all $x\in B$
$$
\lim_{N\to+\infty} a_N(x) \=
\int_B
\max_{\vec v\in\mathcal{V}_{x}\setminus \vec 0}
\left|\log\frac{\|\vec v\|_{2\,(x)}}{\|\vec v\|_{1\,(x)}}\right|\,
d\mu(x) \,<\, \infty\,,
$$
which implies that for almost all $x\in B$ the vanishing of  the limits
\bes
\lim_{N\to+\infty}(a_N-a_{N-1}) \=0\quad \text{and} \quad 
\lim_{N\to+\infty}\frac{1}{N} a_{N-1}\=0\,,
\ees
and hence
\begin{equation}
\label{eq:lim:1:N:equals:0}
\lim_{N\to+\infty} \frac{1}{N}
\max_{\vec v\in\mathcal{V}_{T^{N-1}(x)}\setminus \vec 0}
\left|\log\frac{\|\vec v\|_{2\,(T^N(x))}}{\|\vec v\|_{1\,(T^N(x))}}\right|
\=
\lim_{N\to+\infty} a_N-\frac{N-1}{N}a_{N-1}
\=0\,.
\end{equation}
Thus, for almost any $x\in B$ and for any
$\vec v\in\mathcal{V}_{x}\setminus \vec 0$ we have
\begin{multline*}
\lambda_{(1)}(\vec v)=
\lim_{N\to+\infty} \frac{1}{N}\log\|T^N\vec v(x)\|_1=
\lim_{N\to+\infty} \frac{1}{N}\log
\left(\frac{\|T^N\vec v(x)\|_1}{\|T^N\vec v(x)\|_2}\cdot \|T^N\vec v(x)\|_2\right)
\\=
\lim_{N\to+\infty} \frac{1}{N}\log\frac{\|T^N\vec v(x)\|_1}{\|T^N\vec v(x)\|_2}+
\frac{1}{N}\log\|T^N\vec v(x)\|_2
=
0+\lambda_{(2)}(\vec v)\,,
\end{multline*}
where $\lambda_{(1)}(\vec v)$ (respectively $\lambda_{(2)}(\vec v)$) is the Lyapunov exponent
associated to the vector $\vec v$ defined by the first (respectively by second) cocycle, and
where the equality
$$
\lim_{N\to+\infty} \frac{1}{N}\log\frac{\|T^N\vec v(x)\|_1}{\|T^N\vec v(x)\|_2}\=0
$$
is the corollary of~\eqref{eq:lim:1:N:equals:0}.
\end{proof}
\end{appendix} 

\pagestyle{plain}

\printbibliography

\end{document}